\documentclass{article}

\usepackage{graphicx} 
\usepackage{amsmath,amssymb,bm,amsthm,mathrsfs,amscd,authblk}
\usepackage{here}

\usepackage{fancyhdr}
\usepackage{tikz}
\usepackage{color}
\theoremstyle{plain}
\newtheorem{Them}{Theorem}[section]
\newtheorem{Lem}[Them]{Lemma}
\newtheorem{Cor}[Them]{Corollary}

\newtheorem{Que}[Them]{Question}
\theoremstyle{definition}
\newtheorem{Defi}[Them]{Definition}
\theoremstyle{remark}
\newtheorem{rem}[Them]{Remark}

\numberwithin{equation}{section}

\newcommand{\Ac}{\mathcal{A}}
\newcommand{\Bc}{\mathcal{B}}
\newcommand{\Cc}{\mathcal{C}}
\newcommand{\Dc}{\mathcal{D}}

\newcommand{\ZZ}{\mathbb{Z}}

\title{Characteristic quasi-polynomials of deletions of Shi arrangements of type C and type D}
\author{Akihiro Higashitani \thanks{Graduate School of Information Science and Techonology, Osaka University, Suita 565-0871, Japan. \\
Email: higashitani@ist.osaka-u.ac.jp} \and Masato Konoike \thanks{Graduate School of Information Science and Techonology, Osaka University, Suita 565-0871, Japan. \\ Email: kounoike-m@ist.osaka-u.ac.jp} \and Norihiro Nakashima \thanks{Department of Mathematics, Nagoya Institute of Technology, Aichi 466-8555, Japan. \\ Email: nakashima@nitech.ac.jp} \and Satoshi Ono \thanks{Department of Mathematics, Nagoya Institute of Technology, Aichi 466-8555, Japan. \\ Email: s.ono.341@nitech.jp}}
\date{}

\begin{document}

\maketitle

\begin{abstract}
Characteristic quasi-polynomials enumerate the number of points in the complement of hyperplane arrangements modulo positive integers. 
In this paper, we compute the characteristic quasi-polynomials of the restrictions of the Shi arrangements of type C and type D by one given hyperplane, respectively. 
The case of type C is established by extending the method developed in our previous work on type B (\cite{HN2024}), 
while the case of type D is deduced through a direct connection with the results on type B. 
As a corollary, we determine whether period collapse occurs in the characteristic quasi-polynomials of the deletions of the Shi arrangements of type C and type D. 
Moreover, it turns out that the intersection posets of the Shi arrangements of type B and type C are not isomorphic as posets if their rank is at least $3$. 

\end{abstract}

\tableofcontents

\section{Introduction}

In the study of hyperplane arrangements, the characteristic polynomial of a hyperplane arrangement $\mathcal{A}$ is a fundamental invariant that encodes various algebraic and combinatorial aspects of $\mathcal{A}$. 
Kamiya, Takemura, and Terao \cite{KTT08, KTT11} introduced its natural generalization, now known as the characteristic quasi-polynomial. 
They showed that, for a given arrangement $\mathcal{A}$, the number of points in its complement when considered over $(\mathbb{Z}/q\mathbb{Z})^m$ is given by a quasi-polynomial in $q$. The precise definition will be recalled below. 


Let $q\in\mathbb{Z}_{>0}$ and define $\mathbb{Z}_q:=\mathbb{Z}/q\mathbb{Z}$. 
For $a\in\mathbb{Z}$, let $[a]_q:=a+q\mathbb{Z}\in\mathbb{Z}_q$ be the $q$ reduction of $a$.
Note that as usual, we choose representatives $[a]_q \in \ZZ_q$ as integers with $0 \leq a \leq q-1$, and we often identify elements of $\ZZ_q$ with $\{a \in \ZZ \mid 0 \leq a \leq q-1\}$ in the sequel.   
An $m\times n$ integer matrix $A=(a_{ij})\in\mathrm{Mat}_{m\times n}(\mathbb{Z})$ and an integral vector $\bm{b}=(b_1,\dots,b_n)\in\mathbb{Z}^n$ define a hyperplane arrangement $\Ac_q=\{H_1[q],\dots,H_n[q]\}$ over $\mathbb{Z}_q$, where $H_i[q]=\{(x_1,\dots,x_m)\in\mathbb{Z}_q\mid [a_{1,i}]_q x_1+\cdots+[a_{m,i}]_q x_m=[b_i]_q\}$ for $i=1,\ldots,n$. 
The set $M(\mathcal{A}_q)$ is defined as the complement of $\mathcal{A}_q$, i.e., 
\begin{align*}
M(\mathcal{A}_q)&:=\mathbb{Z}_q^m\setminus\bigcup_{i=1}^n H_i[q] \\
&=\{(x_1,\dots,x_m)\in\mathbb{Z}_q^m\mid [a_{1,i}]_q x_1+\cdots+[a_{m,i}]_q x_m\neq [b_i]_q\ \text{for}\ i=1,\ldots,n\}.
\end{align*}
Kamiya, Takemura, and Terao proved in \cite{KTT08, KTT11} that the function $|M(\mathcal{A}_q)|$ becomes a monic quasi-polynomial, which we call the \textit{characteristic quasi-polynomial} of a hyperplane arrangement $\Ac$, with a period $\rho_A$, called the \textit{lcm period} of $\Ac$, where $\rho_A$ is a certain positive integer defined from the integer matrix $A$. 
(For the details of quasi-polynomials and their lcm periods, consult, e.g., \cite[Subsection 2.1]{HN2024}.)

Since the introduction of characteristic quasi-polynomials of hyperplane arrangements, those have been computed for several important classes (e.g. \cite{MN2024, Yoshinaga2018}) . 
We focus on the characteristic quasi-polynomials of Shi arrangements, which are the central object of this paper. 
\begin{Defi}[Shi arrangements of type B, type C, type D]
Given a positive integer $m$, let $[m]=\{1,2,\ldots,m\}$. 

The Shi arrangement $\mathcal{B}_m$ of type B is defined by
\begin{align*}
\mathcal{B}_m:=\{\{x_i=0\},\{x_i=1\}\mid i\in [m]\}\cup\{\{x_i\pm x_j=0\},\{x_i\pm x_j=1\}\mid 1\leq i<j\leq m\}, 
\end{align*}

The Shi arrangement $\mathcal{C}_m$ of type C is defined by
\begin{align*}
\mathcal{C}_m:=\{\{2x_i=0\},\{2x_i=1\}\mid i\in [m]\}\cup\{\{x_i\pm x_j=0\},\{x_i\pm x_j=1\}\mid 1\leq i<j\leq m\}.
\end{align*}

The Shi arrangement $\Dc_m$ of type D is defined by
\begin{align*}
\Dc_m:=\{\{x_i\pm x_j=0\},\{x_i\pm x_j=1\}\mid 1\leq i<j\leq m\}.
\end{align*}
\end{Defi}

\medskip

In \cite{HN2024}, the first author and the third author compute the characteristic quasi-polynomial of $\Bc_m$ as follows: 
\begin{align}\label{eq:typeB_Shi}
\left| M((\Bc_m)_q)\right| = (q-2m)^m. 
\end{align}
Note that this was essentially obtained in \cite[Theorem 5.1]{Yoshinaga2018}. (See Remark~\ref{rem-basis-transform} below.) 
On the other hand, $\left| M((\Cc_m)_q)\right|$ and $\left| M((\Dc_m)_q)\right|$ are not obtained there, so we compute them in this paper. 
The following is the first main result of this paper: 
\begin{Them}[See Theorems~\ref{thm-chara-quasi-TypeC} and \ref{thm-chara-quasi-TypeD}]\label{thm:Shi} We have 
\begin{equation}\label{eq:typeC_Shi} 
    |M((\Cc_m)_q)|=(q-2m)^m 
\end{equation}
and 
\begin{equation}\label{eq:typeD_Shi}
    |M((\Dc_m)_q)|=(q-2m+2)^m. 
\end{equation}
\end{Them}

\begin{rem}\label{rem-basis-transform}
Yoshinaga \cite[Theorem 5.1]{Yoshinaga2018} determined the characteristic quasi-polynomials for extended Shi arrangements as follows: 
\[\left|M((\mathcal{A}_\Phi^{[1-k,k]})_q)\right| = (q-kh)^m,\]
where $\mathcal{A}_\Phi^{[1-k,k]}$ denotes the extended Shi arrangement and $h$ denotes the Coxeter number. 
In particular, we know that 
\begin{align}\label{eq:Shi}
    \left|M((\mathcal{B}_\Phi^{[0,1]})_q)\right| = \left|M((\mathcal{C}_\Phi^{[0,1]})_q)\right| = (q-2m)^m \;\text{ and }\; \left|M((\mathcal{D}_\Phi^{[0,1]})_q)\right|=(q-2m+2)^m.
\end{align}
In \cite{Yoshinaga2018}, the Shi arrangements are defined as linear combinations of \textit{simple root bases}. 
On the other hand, in this paper, the Shi arrangements are defined as linear combinations of \textit{standard bases}. 
In the case of real arrangements, where the basis transformations are linear isomorphisms, or in the case of arrangements over $\mathbb{Z}_q$, 
where the basis transformations are unimodular, hyperplane arrangements that are transformed by the basis transformations are isomorphic. 
However, since the determinant of the transformation matrices of the simple roots and the standard basis is $2$ for the Shi arrangements of type C and type D, they are distinct as arrangements over $\mathbb{Z}_q$. 
Therefore, the characteristic quasi-polynomials considered in this paper may not necessarily coincide with that of the Shi arrangements represented by the basis of simple roots. 
More concretely, Theorem~\ref{thm:Shi} is essentially different from \eqref{eq:Shi}, but we eventually notice that those are the same. 
\end{rem}

We say that \textit{period collapse} occurs for a hyperplane arrangement $\mathcal{A}$ if the minimum period of its characteristic quasi-polynomial is strictly smaller than its lcm period. 
In \cite{HTY2023}, 
they claim that “any possible period collapse” can be realized as a non-central arrangement. 
For the (extended) Shi arrangements of type B, type C and type D, the equalities in Theorem~\ref{thm:Shi} imply that the characteristic quasi-polynomial reduces to an ordinary polynomial (minimum period $1$), whereas the lcm period is known to be $2$. 
Therefore, period collapse occurs in each of these arrangements.

\medskip

Motivated by these, we investigate the following question:
\begin{Que}[{See \cite[Question 1.2]{HN2024}}]\label{q}
Let $\mathcal{A}$ be a hyperplane arrangement and fix $H \in \mathcal{A}$. Under what conditions does period collapse occur in the deletion $\mathcal{A} \setminus {H}$?
\end{Que}
We completely answered this question for the Shi arrangement of type B in \cite{HN2024} as shown in Theorem~\ref{typeB}. 
Our goal of this paper is to answer this question completely for the Shi arrangements of type C or type D. 
For this purpose, we compute the characteristic quasi-polynomials of the restrictions $\mathcal{C}_m^H$ and $\mathcal{D}_m^H$ obtained by fixing a hyperplane $H$ in the corresponding arrangement. In the case of type C, the proof parallels the approach used in our previous work on type B, whereas for type D we exploit a close relationship with the type B result. 
For a hyperplane arrangement $\Ac$ defined by an integer
matrix and an integral vector as above and $H\in\Ac$,
we define
\begin{align*}
M(\Ac_q^{H_q}):=H_q\setminus\bigcup_{K\in\Ac\setminus\{H\}}(H_q\cap K_q).    
\end{align*}
Since
\begin{align*}
    M((\Ac\setminus\{H\})_q)=M(\Ac_q)\sqcup M(\Ac_q^{H_q}),
\end{align*}
we obtain the deletion--restriction formula
\begin{align}\label{eq:del-res}
|M(\Ac_q)|=|M((\Ac\setminus\{H\})_q)|-|M(\Ac_q^{H_q})|.
\end{align}
This formula holds for every positive integer $q$.
Throughout this paper, we also write $M((\Ac^H)_q)$ for $M(\Ac_q^{H_q})$.
See also \cite[Corollary~4.2]{MN2024} for a related deletion--restriction formula.
Namely, the computation of the characteristic quasi-polynomial of the deletion $\Ac \setminus \{H\}$ is equivalent to that of the restriction $\Ac^H$ if we know the characteristic quasi-polynomial of $\Ac$. 
In particular, when the characteristic quasi-polynomial of $\Ac$ is a polynomial, by checking whether the characteristic quasi-polynomial of a restriction is a polynomial, we can check whether the characteristic quasi-polynomial of a deletion is a polynomial.

We recall the following previous result, proved by the first and third authors. 
\begin{Them}[{\cite[Theorem 1.3]{HN2024}}]\label{typeB}
Fix $i$ and $j$ with $1\leq i<j \leq m$. Then 
\begin{align*}
\left|M((\Bc_m^{\{x_i=0\}})_q)\right|&= (T+1)^{m-i}(T+2)^{i-1}; \\
\left|M((\Bc_m^{\{x_i=1\}})_q)\right|&= T^{i-1}(T+1)^{m-i}; \\
\left|M((\Bc_m^{\{x_i-x_j=0\}})_q)\right|&=
\begin{cases}
(T+1)^{j-i-1}(T+2)^{m-j}((T+2)^i-(T+1)^{i-1}) \;\;&\text{if $q$ is odd}, \\
(T+1)^{j-i-1}(T+2)^{i-1}((T+2)^{m-j+1}-(T+1)^{m-j}) &\text{if $q$ is even}; 
\end{cases} \\
\left|M((\Bc_m^{\{x_i-x_j=1\}})_q)\right|&=T^{m+i-j}(T+1)^{j-i-1}; \\
\left|M((\Bc_m^{\{x_i+x_j=0\}})_q)\right|&=
\begin{cases}
T^{m-j}(T+1)^{j-i}((T+2)^{i-1}-(T+1)^{i-2}) \;\;&\text{if $q$ is odd}, \\
T^{m-j+1}(T+1)^{j-i-1}(T+2)^{i-1} &\text{if $q$ is even}; 
\end{cases} \\
\left|M((\Bc_m^{\{x_i+x_j=1\}})_q)\right|&=
\begin{cases}
T^{i-1}(T+1)^{j-i}(T+2)^{m-j} \;\;&\text{if $q$ is odd}, \\
T^{i-1}(T+1)^{j-i-1}((T+2)^{m-j+1}-(T+1)^{m-j}) &\text{if $q$ is even}; 
\end{cases}
\end{align*}
where $T:=q-2m$. 
\end{Them}

The following theorems are our main results. 
\begin{Them}[See Theorems~\ref{thm-char-quasi-TypeC-2xi=0}, \ref{thm-char-quasi-TypeC-2xi=1}, \ref{thm-char-quasi-TypeC-xi=xj}, \ref{thm-char-quasi-TypeC-xi=xj+1}, \ref{thm-char-quasi-TypeC-xi=-xj} and \ref{thm-char-quasi-TypeC-xi=-xj+1}]\label{thm-main-typeC}
Fix $i$ and $j$ with $1\leq i<j \leq m$. Then 
\begin{align*}
\left|M({(\mathcal{C}_m^{\{2x_i=0\}})}_q)\right|&=
\begin{cases}
T^{m-i}(T+1)^{i-1}\;\;&\text{if $q$ is odd},\\
2T^{m-i}(T+1)^{i-1} &\text{if $q$ is even};\\
\end{cases} \\
\left|M({(\mathcal{C}_m^{\{2x_i=1\}})}_q)\right|&=
\begin{cases}
T^{i-1}(T+1)^{m-i}\;\;\;&\text{if $q$ is odd},\\
0 &\text{if $q$ is even};\\
\end{cases} \\
\left|M{((\mathcal{C}_m^{\{x_i - x_j=0\}})}_q)\right|&=(T+1)^{j-i-1}(T+2)^{m-j+i};\\
\left|M({(\mathcal{C}_m^{\{x_i - x_j=1\}})}_q)\right|&=T^{m-j+i}(T+1)^{j-i-1};\\
\left|M({(\mathcal{C}_m^{\{x_i + x_j=0\}})}_q)\right|&=
\begin{cases}
T^{m-j}(T+1)^{j-i}(T+2)^{i-1} \;\; &\text{if $q$ is odd},\\
T^{m-j}(T+1)^{j-i-1}(T+2)^i &\text{if $q$ is even};
\end{cases} \\
\left|M({(\mathcal{C}_m^{\{x_i + x_j=1\}})}_q)\right|&=
\begin{cases}
T^{i-1}(T+1)^{j-i}(T+2)^{m-j} \;\; &\text{if $q$ is odd},\\
T^i(T+1)^{j-i-1}(T+2)^{m-j} &\text{if $q$ is even}; 
\end{cases} 
\end{align*}
where $T:=q-2m$. 
\end{Them}

\begin{Them}[See Theorems~\ref{chara-quasi-del-by-xi=-xj}, \ref{chara-quasi-del-by-xi=xj+1}, \ref{chara-quasi-del-by-xi=xj} and \ref{chara-quasi-del-by-xi=-xj+1}]\label{thm-main-typeD}
Fix $i$ and $j$ with $1\leq i<j \leq m$. Then 
    \begin{align*}
    \left|M((\Dc_m^{\{x_i-x_j=0\}})_q)\right|&=
    \begin{cases}
        (T+3)^{j-i-1}(T+4)^{m-j}((T+4)^i-(T+3)^{i-1})\\
        -(T+3)^{m-i-1}(T+4)^{i-1}\;\;\;& \text{if $q$ is odd},\\
        (T+3)^{j-i-1}(T+4)^{i-1}((T+4)^{m-j+1}-(T+3)^{m-j}) \\
        -(T+3)^{m-i-1}(T+4)^{i-1} & \text{if $q$ is even}; 
    \end{cases} \\
    \left|M((\Dc_m^{\{x_i-x_j=1\}})_q)\right|&=
    (T+2)^{m+i-j}(T+3)^{j-i-1}; \\
    \left|M((\Dc_m^{\{x_i+x_j=0\}})_q)\right|&=
    \begin{cases}
        (T+2)^{m-j}(T+3)^{j-i}((T+4)^{i-1}-(T+3)^{i-2}) \;\;&\text{if $q$ is odd}, \\
        (T+2)^{m-j+1}(T+3)^{j-i-1}(T+4)^{i-1} & \text{if $q$ is even}; 
    \end{cases} \\
    \left|M((\Dc_m^{\{x_i+x_j=1\}})_q)\right|&=
    \begin{cases}
        (T+2)^{i-1}(T+3)^{j-i}(T+4)^{m-j} - (T+2)^{i-1}(T+3)^{m-i-1} \;\;& \hspace{-0.3cm}\text{if $q$ is odd}, \\
    (T+2)^{i-1}(T+3)^{j-i-1}((T+4)^{m-j+1}-(T+3)^{m-j}) \\
    - (T+2)^{i-1}(T+3)^{m-i-1}& \hspace{-0.3cm} \text{if $q$ is even}; 
    \end{cases}
\end{align*}
where $T:=q-2m$. 
\end{Them}

As corollaries of Theorems~\ref{thm-main-typeC} and \ref{thm-main-typeD}, we obtain the following: 
\begin{Cor}\label{period-collapse-1}
{\em (1)} Fix $H \in \Cc_m$. 
Then the characteristic quasi-polynomial of $\Cc_m \setminus \{H\}$ becomes a polynomial if and only if $H$ is one of the following: 
\begin{itemize}
    \item $H=\{x_i-x_j=0\}$ for $1 \leq i<j \leq m$;
    \item $H=\{x_i-x_j=1\}$ for $1 \leq i<j \leq m$.
\end{itemize}
{\em (2)} Fix $H \in \Dc_m$. 
Then the characteristic quasi-polynomial of $\Dc_m \setminus \{H\}$ becomes a polynomial if and only if $H$ is one of the following: 
\begin{itemize}
    \item $H=\{x_i-x_{m+1-i}=0\}$ for $1\leq i\leq\lfloor m/2\rfloor$;
    \item $H=\{x_i-x_j=1\}$ for $1 \leq i<j \leq m$; 
    \item $H=\{x_1+x_j=0\}$ for $2 \leq j \leq m$; 
    \item $H=\{x_i+x_m=1\}$ for $1 \leq i \leq m-1$.  
\end{itemize}
\end{Cor}

\begin{Cor}\label{period-collapse-2}
{\em (1)} Fix $H,H^{\prime} \in \Cc_m$ which are parallel each other. 
Then the characteristic quasi-polynomial of $\Cc_m \setminus \{H,H^{\prime}\}$ becomes a polynomial if and only if one of the following is satisfied: 
\begin{itemize}
    \item $H=\{2x_{(m+1)/2}=0\}$ and $H^{\prime}=\{2x_{(m+1)/2}=1\}$, where $m$ is odd;
    \item $H=\{x_i-x_j=0\}$ and $H^{\prime}=\{x_i-x_j=1\}$ for $1 \leq i <j \leq m$; 
    \item $H=\{x_i+x_{m+1-i}=0\}$ and $H^{\prime}=\{x_i+x_{m+1-i}=1\}$ for $1\leq i\leq\lfloor m/2\rfloor$. 
\end{itemize}
{\em (2)} Fix $H,H^{\prime} \in \Dc_m$ which are parallel each other. 
Then the characteristic quasi-polynomial of $\Dc_m \setminus \{H,H^{\prime}\}$ becomes a polynomial if and only if one of the following is satisfied: 
\begin{itemize}
    \item $H=\{x_i-x_{m+1-i}=0\}$ and $H^{\prime}=\{x_i-x_{m+1-i}=1\}$ for $1\leq i\leq\lfloor m/2\rfloor$; 
    \item $H=\{x_i+x_{m+1-i}=0\}$ and $H^{\prime}=\{x_i+x_{m+1-i}=1\}$ for $1\leq i\leq\lfloor m/2\rfloor$. 
\end{itemize}
\end{Cor}

It is known that, in general, the characteristic polynomial of a hyperplane arrangement coincides with the first constituent of its characteristic quasi-polynomial (see \cite[Theorem 2.1]{Athanasiadis1999}).
For Shi arrangements of type B and type C, this constituent corresponds to the case where $q$ is odd.
By calculating the characteristic polynomials of restriction arrangements for all hyperplanes in Theorems~\ref{typeB} and \ref{thm-main-typeC}, we can see the following.
\begin{Cor}\label{Cor-latticeBC}
   If $m \geq 3$, then the intersection posets of the Shi arrangements of type B and type C are not isomorphic as posets.
\end{Cor}


\medskip

This paper is organized as follows. 
In Section~\ref{sec-TypeC}, we give a proof of Theorem~\ref{thm-main-typeC} by modifying the counting method, as performed for the Shi arrangement of type B in \cite{HN2024}. 
In Section~\ref{sec-TypeD}, we give a proof of Theorem~\ref{thm-main-typeD} by constructing a certain bijection between $M(\Dc_q^H)$ and $M(\Bc_{q+2}^H)$. 
For both of type C and type D, as an appetizer, we provide a proof of the characteristic quasi-polynomial of $\Cc_m$ and $\Dc_m$ in the beginning of Sections~\ref{sec-TypeC} and \ref{sec-TypeD}, respectively. 
Finally, in Section~\ref{sec:corollaries}, we give proofs of Corollaries~\ref{period-collapse-1}, \ref{period-collapse-2} and \ref{Cor-latticeBC}.

\section*{Acknowledgements}
A.H. is partially supported by KAKENHI JP24K00521 and JP23H00081.


\section{Shi arrangement of type C}\label{sec-TypeC}
The characteristic quasi-polynomials of the Shi arrangement of type C can be computed in the same way as those of type B described in \cite{HN2024}, using a modified version of the method obtained by Athanasiadis \cite{Athanasiadis1996}.

\subsection{The counting method}\label{sec-TypeC-counting}
Let $\Cc=\mathcal{C}_m$ be the Shi arrangement of type C.
In this subsection, we prove that $|M(\Cc_q)|=(q-2m)^m$ for any $q\gg 0$ by dividing the cases where $q$ is odd or even.
The complement $M(\Cc_q)$ is the set of elements $(x_1,\dots,x_m)\in\mathbb{Z}_q^m$ that satisfies the following conditions:
\begin{align*}
&2x_s\neq 0,\ 2x_s\neq 1\ (s\in [m]),\\
&x_s\neq x_t\ (s,t\in [m],\ s\neq t),\\
&x_s\neq x_t+1\ (s,t\in [m],\ s<t),\\
&x_s\neq -x_t,\ x_s\neq -x_t+1\ (s,t\in [m],\ s\neq t).
\end{align*}
In this subsection, we prove the following.
\begin{Them}\label{thm-chara-quasi-TypeC}
The characteristic quasi-polynomial of the Shi arrangement of type C is 
\begin{align*}
\left|M(\Cc_q)\right|=(q-2m)^m
\end{align*}
for any $q\in\mathbb{Z}$ with $q\gg 0$.
\end{Them}
Note that the result of Theorem \ref{thm-chara-quasi-TypeC} coincides with \eqref{eq:Shi}, but a different one is computed in this paper (see Remark~\ref{rem-basis-transform}).

\subsubsection{The case where $q$ is odd}\label{sec-TypeC-odd}
In this subsection, we assume that $q$ is odd and count the number of elements $(x_1,\dots,x_m)\in M(\Cc_q)$ by creating boxes and circles similar to \cite[Section 2..4.1]{HN2024}.
To aid in understanding, we give an example for $m=5$ and $q=15$.

\begin{itemize}
\item[1.] Prepare $q-2m+1$ boxes side by side, $\displaystyle \frac{q+1}{2}-m$ on the upper side and $\displaystyle \frac{q+1}{2}-m$ on the lower side.
Then place each of the numbers $1,\dots,m$ corresponding to the indices of $x_1,\dots,x_m$ in one of $q-2m$ boxes, avoiding the lower right box.
Note that in the case of type B, the upper left box is avoided, but in the case of type C and $q$ is odd, the lower right box is avoided.
This is because the conditions $2x_s\neq 1$ are equivalent to $\displaystyle x_s\neq \frac{q+1}{2}$, respectively.
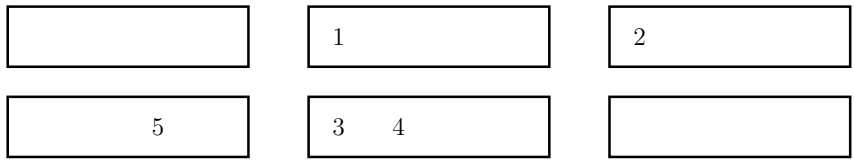
\begin{center}
\scalebox{0.8}{
\begin{tikzpicture}[main/.style = {draw, circle, very thick}] 
\draw [very thick] (0,0) rectangle (4,1); \draw [very thick] (5,0) rectangle (9,1); \draw [very thick] (10,0) rectangle (14,1);
\draw [very thick] (0,1.5) rectangle (4,2.5); \draw [very thick] (5,1.5) rectangle (9,2.5); \draw [very thick] (10,1.5) rectangle (14,2.5);
\node(1) at (5.5,2) {{\large $1$}}; \node(2) at (10.5,2) {{\large $2$}}; \node(3) at (2.5,0.5) {{\large $5$}}; \node(4) at (5.5,0.5) {{\large $3$}}; \node(5) at (6.5,0.5) {{\large $4$}};
\end{tikzpicture}
}
\end{center}
\item[2.]  Place unlabeled circles at the left edges of all boxes, except the lower left box.
Rewrite each number as a circle labeled with the same number.
The labeled circles in each of the upper boxes are placed in ascending order from left to right, next to the unlabeled circle.
Place the unlabeled circles on the opposite side of the labeled circles arranged in this manner.
After that, the labeled circles in each of the lower boxes are placed in descending order, starting from next of unlabeled circles placed in this way.
Also, place the unlabeled circles on the opposite side.
\item[3.] Arrange the circles clockwisely starting from the circle at the left end of the upper left box.
The left end circle of the upper left box corresponds to $0$, and other circles clockwisely correspond to the elements in $\{1,\dots,q-1\}$ in ascending order.
Create a tuple whose $i$-th entry is the element in $\{1,\dots,q-1\}$ corresponding to the circle labeled with $i$.
Then we obtain an element $(x_1,\dots,x_m)\in\mathbb{Z}_q^m$.
In the case of the example, $(x_1,x_2,x_3,x_4,x_5)=(3,7,10,11,14)$: 
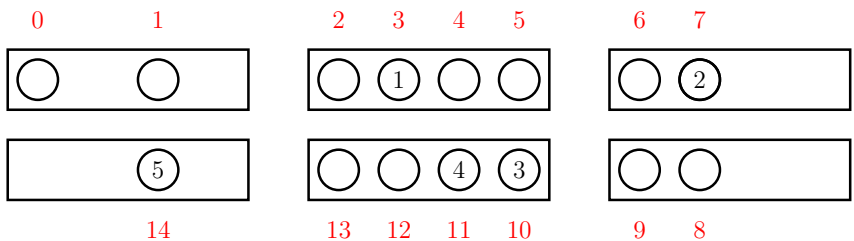
\begin{center}
\scalebox{0.8}{
\begin{tikzpicture}[main/.style = {draw, circle, very thick}] 
\draw [very thick] (0,0) rectangle (4,1); \draw [very thick] (5,0) rectangle (9,1); \draw [very thick] (10,0) rectangle (14,1);
\draw [very thick] (0,1.5) rectangle (4,2.5); \draw [very thick] (5,1.5) rectangle (9,2.5); \draw [very thick] (10,1.5) rectangle (14,2.5);
\node[main](1) at (0.5,2) {{\color{white}\large $0$}}; \node[main](2) at (2.5,2) {{\color{white}\large $0$}};
\node[main](3) at (5.5,2) {{\color{white}\large $0$}}; \node[main](4) at (6.5,2) {{\large $1$}};\node[main](5) at (7.5,2) {{\color{white}\large $0$}}; \node[main](6) at (8.5,2) {{\color{white}\large $0$}};
\node[main](7) at (10.5,2) {{\color{white}\large $0$}}; \node[main](8) at (11.5,2) {{\color{white}\large $0$}};
\node[main](9) at (2.5,0.5) {{\large $5$}};
\node[main](10) at (5.5,0.5) {{\color{white}\large $0$}}; \node[main](11) at (6.5,0.5) {{\color{white}\large $0$}};\node[main](12) at (7.5,0.5) {{\large $4$}}; \node[main](13) at (8.5,0.5) {{\large $3$}};
\node[main](15) at (11.5,2) {{\large $2$}};\node[main](16) at (11.5,0.5) {{\color{white}\large $0$}};\node[main](16) at (10.5,0.5) {{\color{white}\large $0$}};
\node(0) at (0.5,3) {{\color{red}\large $0$}}; \node(1) at (2.5,3) {{\color{red}\large $1$}}; \node(2) at (5.5,3) {{\color{red}\large $2$}}; \node(3) at (6.5,3) {{\color{red}\large $3$}}; \node(4) at (7.5,3) {{\color{red}\large $4$}}; \node(5) at (8.5,3) {{\color{red}\large $5$}}; \node(6) at (10.5,3) {{\color{red}\large $6$}}; \node(7) at (11.5,3) {{\color{red}\large $7$}};
\node(8) at (11.5,-0.5) {{\color{red}\large $8$}}; \node(9) at (10.5,-0.5) {{\color{red}\large $9$}}; \node(10) at (8.5,-0.5) {{\color{red}\large $10$}}; \node(11) at (7.5,-0.5) {{\color{red}\large $11$}}; \node(12) at (6.5,-0.5) {{\color{red}\large $12$}}; \node(13) at (5.5,-0.5) {{\color{red}\large $13$}}; \node(14) at (2.5,-0.5) {{\color{red}\large $14$}};
\end{tikzpicture}
}
\end{center}
\end{itemize}
The tuple created by the above method satisfies the following conditions and is an element of $M(\Cc_q)$.
\begin{itemize}
\item The circle corresponding to the element $0\in\mathbb{Z}_q$ is always unlabeled and this corresponds to the condition $2x_s\neq 0$ ($\Leftrightarrow$ $x_s\neq 0$) for any $s\in [m]$.
The circle corresponding to the element $\displaystyle \frac{q+1}{2}\in\mathbb{Z}_q$ has no label since we avoid putting the numbers $1,\dots,m$ in the lower right box.
This corresponds to the condition $2x_s\neq 1$ for any $s\in [m]$.
\item Each circle is labeled with at most one number.
This corresponds to the condition that $x_s\neq x_t$ for any $s,t\in [m]$ with $s\neq t$.
\item The opposite circle of the labeled circle is always unlabeled.
This corresponds to the condition that $x_s\neq -x_t$ for any $s,t\in [m]$ with $s\neq t$.
\item The circle preceding the circle with the label $s$ in the clockwise direction does not have a label greater than $s$.
This corresponds to the condition that $x_s\neq x_t+1$ for any $s,t\in [m]$ with $s<t$.
\item The clockwise next circle from the opposite circle of the circle with the label $t$ is either an unlabeled circle or the circle with the label $t$ (that is, the same circle as the original).
This corresponds to the condition that $x_t\neq -x_s+1$ for any $s,t\in [m]$ with $s\neq t$.
\end{itemize}
Note that the conditions after the second item are the same as those for Type B.
Following this procedure in reverse, taking an element in $M(\Cc_q)$ corresponds to placing the numbers $1,\dots,m$ in the $q-2m$ boxes, except for the lower right box.
Therefore, we have $|M(\Cc_q)|=(q-2m)^m$.

\subsubsection{The case where $q$ is even}\label{sec-TypeC-even}
In this subsection, we assume that $q$ is even and extend the counting method to the case where $q$ is even.
Prepare $q-2m$ boxes side by side, $\displaystyle \frac{q}{2}-m$ on the upper side and $\displaystyle \frac{q}{2}-m$ on the lower side, and a circle corresponding to the element $\displaystyle \frac{q}{2}\in\mathbb{Z}_q$ on the right side of the boxes.
From the condition $2x_s\neq 0$ for any $s\in [m]$, there exists no index $s$ such that $\displaystyle x_s=\frac{q}{2}$.
In other words, the circle corresponding to $\displaystyle \frac{q}{2}\in\mathbb{Z}_q$ is unlabeled.
In addition, unlike the case where $q$ is odd, there does not exist $x_s\in\mathbb{Z}_q$ such that $2x_s=1$, so the numbers can also be placed in the lower right box.
Then, the numbers $1,\dots,m$ are placed in $q-2m$ boxes to create circles and boxes.
The boxes and circles created in this way correspond one-to-one to the elements of $M(\Cc_q)$ in the same way as in Section \ref{sec-TypeC-odd}.
Therefore, we have $|M(\Cc_q)|=(q-2m)^m$ in this case.
The following is an example of the boxes and circles corresponding to the element $(x_1,x_2,x_3,x_4,x_5)=(3,7,11,12,15)\in M(\Cc_q)$ in the case of $m = 5, q = 16$:
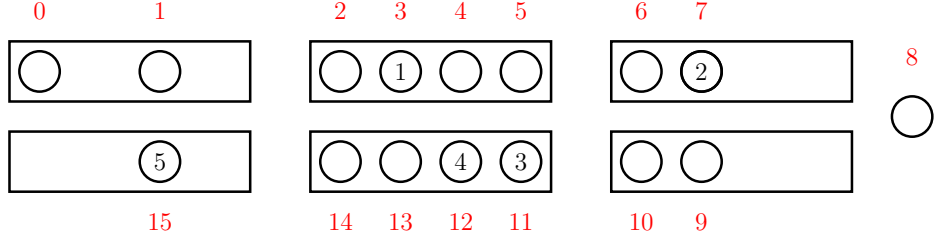
\begin{center}
\scalebox{0.8}{
\begin{tikzpicture}[main/.style = {draw, circle, very thick}] 
\draw [very thick] (0,0) rectangle (4,1); \draw [very thick] (5,0) rectangle (9,1); \draw [very thick] (10,0) rectangle (14,1);
\draw [very thick] (0,1.5) rectangle (4,2.5); \draw [very thick] (5,1.5) rectangle (9,2.5); \draw [very thick] (10,1.5) rectangle (14,2.5);
\node[main](1) at (0.5,2) {{\color{white}\large $0$}}; \node[main](2) at (2.5,2) {{\color{white}\large $0$}};
\node[main](3) at (5.5,2) {{\color{white}\large $0$}}; \node[main](4) at (6.5,2) {{\large $1$}};\node[main](5) at (7.5,2) {{\color{white}\large $0$}}; \node[main](6) at (8.5,2) {{\color{white}\large $0$}};
\node[main](7) at (10.5,2) {{\color{white}\large $0$}}; \node[main](8) at (11.5,2) {{\color{white}\large $0$}};
\node[main](9) at (2.5,0.5) {{\large $5$}};
\node[main](10) at (5.5,0.5) {{\color{white}\large $0$}}; \node[main](11) at (6.5,0.5) {{\color{white}\large $0$}};\node[main](12) at (7.5,0.5) {{\large $4$}}; \node[main](13) at (8.5,0.5) {{\large $3$}};
\node[main](15) at (11.5,2) {{\large $2$}};\node[main](16) at (11.5,0.5) {{\color{white}\large $0$}};\node[main](16) at (10.5,0.5) {{\color{white}\large $0$}};\node[main](17) at (15,1.25) {{\color{white}\large $0$}};
\node(0) at (0.5,3) {{\color{red}\large $0$}}; \node(1) at (2.5,3) {{\color{red}\large $1$}}; \node(2) at (5.5,3) {{\color{red}\large $2$}}; \node(3) at (6.5,3) {{\color{red}\large $3$}}; \node(4) at (7.5,3) {{\color{red}\large $4$}}; \node(5) at (8.5,3) {{\color{red}\large $5$}}; \node(6) at (10.5,3) {{\color{red}\large $6$}}; \node(7) at (11.5,3) {{\color{red}\large $7$}};
\node(8) at (11.5,-0.5) {{\color{red}\large $9$}}; \node(9) at (10.5,-0.5) {{\color{red}\large $10$}}; \node(10) at (8.5,-0.5) {{\color{red}\large $11$}}; \node(11) at (7.5,-0.5) {{\color{red}\large $12$}}; \node(12) at (6.5,-0.5) {{\color{red}\large $13$}}; \node(13) at (5.5,-0.5) {{\color{red}\large $14$}}; \node(14) at (2.5,-0.5) {{\color{red}\large $15$}};\node(17) at (15,2.25) {{\color{red}\large $8$}};
\end{tikzpicture}
}
\end{center}

\subsection{Characteristic quasi-polynomial of restriction on $\{2x_i=0\}$}\label{sec-TypeC-rest-2x_i=0}
Let $\Cc^{(1)}=\mathcal{C}_m^{\{2x_i=0\}}$ for $i\in [m]$, and we prove the first equality of Theorem \ref{thm-main-typeC} as follows.
Here and throughout, $M(\Cc_q^{(1)})$ denotes the complement of the restriction defined above, with $\mathcal{A}=\Cc_m$ and $H=\{2x_i=0\}$.
We use analogous notation for the other restrictions below.
\begin{Them}\label{thm-char-quasi-TypeC-2xi=0}
We have
\begin{align*}
\left|M(\Cc_q^{(1)})\right|=
\begin{cases}
(q-2m)^{m-i}(q-2m+1)^{i-1}\;\;\;&\text{if $q$ is odd}, \\
2(q-2m)^{m-i}(q-2m+1)^{i-1} &\text{if $q$ is even}
\end{cases}
\end{align*}
for any $q\in\mathbb{Z}$ with $q\gg 0$.
\end{Them}
The complement $M(\Cc^{(1)}_q)$ is the set of $(x_1,\dots,x_m)\in\mathbb{Z}_q^m$ that satisfies the following conditions:
\begin{align*}
&2x_i=0,\\
&2x_s\neq 0,\ 2x_s\neq 1\ (s\in [m],\ s\neq i),\\
&x_s\neq x_t\ (s,t\in [m],\ s\neq t),\\
&x_s\neq x_t+1\ (s,t\in [m],\ s<t),\\
&x_s\neq -x_t,\ x_s\neq -x_t+1\ (s,t\in [m],\ s\neq t).
\end{align*}
We fix the condition $2x_i=0$ and count the number of elements in $(x_1,\dots,x_m)\in M(\Cc^{(1)}_q)$ using a modified version of the counting method described in Section \ref{sec-TypeC-counting}.
We have $2x_i=0\ \Leftrightarrow\ x_i=0$ if $q$ is odd and $\displaystyle 2x_i=0\ \Leftrightarrow\ x_i\in \left\{0,\frac{q}{2}\right\}$ if $q$ is even.
Therefore, the condition $2x_i=0$ means that the leftmost circle in the upper left box has the label $i$ if $q$ is odd, and that either the leftmost circle in the upper left box or the circle on the right side of the boxes has the label $i$ if $q$ is even. 

\subsubsection{The case where $q$ is odd}\label{sec-TypeC-rest-2x_i=0-odd}
Let $q$ be odd.
Then we create boxes and circles corresponding to the element $(x_1,\dots,x_m)\in M(\Cc^{(1)}_q)$.
Prepare $q-2m+3$ boxes side by side, $\displaystyle \frac{q+1}{2}-(m-1)$ on the upper side and $\displaystyle \frac{q+1}{2}-(m-1)$ on the lower side.
Place the circle with the label $i$ at the left edge of the upper left box.
Then place each of the numbers $1,\dots,i-1,i+1,\dots,m$ in one of the $q-2m+3$ boxes, where each of the numbers cannot be placed in the lower right box, since $2x_s\neq 1$ for any $s\in [m]$.
In addition, each of the numbers cannot be placed in the upper left box, since $x_s\neq -x_i+1$ for $s\neq i$. 
(Note that $-x_i=x_i$ since $x_i=0$.) 
\begin{itemize}
\item There are $q-2m+1$ boxes that can contain the numbers $1,\dots,i-1$, except the upper left and lower right boxes.

\item Each of the numbers $i+1,\dots,m$ cannot be placed in the lower left box, since $x_i\neq x_s+1$ for $i<s$.
There are $q-2m$ boxes that can contain the numbers $i+1,\dots,m$, except the upper left, lower left, and lower right boxes.
\end{itemize}
Therefore, we have $|M(\Cc^{(1)}_q)|=(q-2m)^{m-i}(q-2m+1)^{i-1}$ in this case.
For example, the following boxes and circles correspond to the element $(x_1,x_2,x_3,x_4,x_5)=(5,10,0,11,6)\in M(\Cc^{(1)}_q)$ in the case of $m=5,q=13,i=3$:
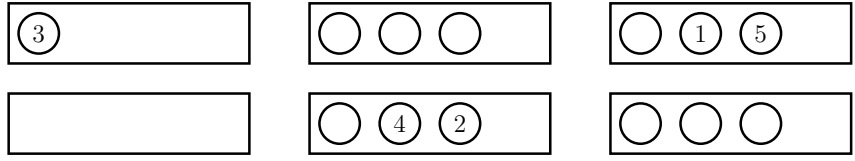
\begin{center}
\scalebox{0.8}{
\begin{tikzpicture}[main/.style = {draw, circle, very thick}] 
\draw [very thick] (0,0) rectangle (4,1); \draw [very thick] (5,0) rectangle (9,1); \draw [very thick] (10,0) rectangle (14,1);
\draw [very thick] (0,1.5) rectangle (4,2.5); \draw [very thick] (5,1.5) rectangle (9,2.5); \draw [very thick] (10,1.5) rectangle (14,2.5);
\node[main](0) at (0.5,2) {{\large $3$}};  \node[main](3) at (12.5,0.5) {{\color{white}\large $0$}};
\node[main](4) at (5.5,2) {{\color{white}\large $0$}}; \node[main](5) at (6.5,2) {{\color{white}\large $0$}};
\node[main](6) at (10.5,2) {{\color{white}\large $0$}}; \node[main](7) at (11.5,2) {{\large $1$}};\node[main](1) at (7.5,2) {{\color{white}\large $0$}};  \node[main](2) at (7.5,0.5) {{\large $2$}};
\node[main](12) at (12.5,2) {{\large $5$}};
\node[main](11) at (5.5,0.5) {{\color{white}\large $0$}}; \node[main](10) at (6.5,0.5) {{\large $4$}};
\node[main](9) at (10.5,0.5) {{\color{white}\large $0$}}; \node[main](8) at (11.5,0.5) {{\color{white}\large $0$}};
\end{tikzpicture}
}
\end{center}

\subsubsection{The case where $q$ is even}\label{sec-TypeC-rest-2x_i=0-even}
Let $q$ be even.
We create boxes and circles corresponding to the element $(x_1,\dots,x_m)\in M(\Cc^{(1)}_q)$.
Since $\displaystyle 2x_i=0\ \Leftrightarrow\ x_i\in \left\{0,\frac{q}{2}\right\}$, we separate the cases $x_i=0$ and $\displaystyle x_i=\frac{q}{2}$.

(i) Let $x_i=0$.
Prepare $q-2m+2$ boxes side by side, $\displaystyle \frac{q}{2}-(m-1)$ on the upper side and $\displaystyle \frac{q}{2}-(m-1)$ on the lower side, and a circle corresponding to the element $\displaystyle \frac{q}{2}\in\mathbb{Z}_q$ on the right side of the boxes.
In this case, the circle at the left edge of the upper left box is the circle with the label $i$.
The circle on the right side of the boxes has no labels, since $2x_s\neq 0$ for $s\neq i$.
Place each of the numbers $1,\dots,i-1,i+1,\dots,m$ in one of the $q-2m+2$ boxes.
\begin{itemize}
\item There are $q-2m+1$ boxes that can contain the numbers $1,\dots,i-1$, except the upper left box.

\item Similarly to the case where $q$ is odd, each of the numbers $i+1,\dots,m$ cannot be placed in the upper left and lower left boxes.
There are $q-2m$ boxes that can contain the numbers $i+1,\dots,m$.
\end{itemize}
The number of elements of $M(\Cc^{(1)}_q)$ in this case is $(q-2m)^{m-i}(q-2m+1)^{i-1}$.

(i\hspace{-0.5mm}i) Let $\displaystyle x_i=\frac{q}{2}$.
Prepare $q-2m+2$ boxes side by side, $\displaystyle \frac{q}{2}-(m-1)$ on the upper side and $\displaystyle \frac{q}{2}-(m-1)$ on the lower side, and a circle with the label $i$ on the right side of the boxes.
If there exists $s\in[m]\setminus\{i\}$ such that $\displaystyle x_s=\frac{q}{2}+1$, then we have $-x_s+1=x_i$, a contradiction. 
Thus, each of the numbers $1,\dots,i-1,i+1,\dots,m$ cannot be placed in the lower right box.
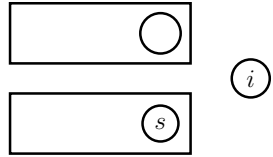
\begin{center}
\scalebox{0.8}{
\begin{tikzpicture}[main/.style = {draw, circle, very thick}] 
\draw [very thick] (0,0) rectangle (3,1); \draw [very thick] (0,1.5) rectangle (3,2.5);
\node[main](0) at (4,1.25) {{\large $i$}};
\node[main](1) at (2.5,2) {{\color{white}\large $0$}}; \node[main](2) at (2.5,0.5) {{\large $s$}};
\end{tikzpicture}
}
\end{center}
\begin{itemize}
\item There are $q-2m+1$ boxes that can contain the numbers $1,\dots,i-1$, except the lower right box.

\item Each of the numbers $i+1,\dots,m$ cannot be placed in the upper right box, since $x_i\neq x_s+1$ for $i<s$.
There are $q-2m$ boxes that can contain the numbers $i+1,\dots,m$, except the upper right and lower right boxes.
\end{itemize}
The number of elements of $M(\Cc^{(1)}_q)$ in this case is also $(q-2m)^{m-i}(q-2m+1)^{i-1}$.
For example, the following boxes and circles correspond to the element $(x_1,x_2,x_3,x_4,x_5)=(1,9,10,7,2)\in M(\Cc^{(1)}_q)$ in the case of $m=5,q=14,i=4$:
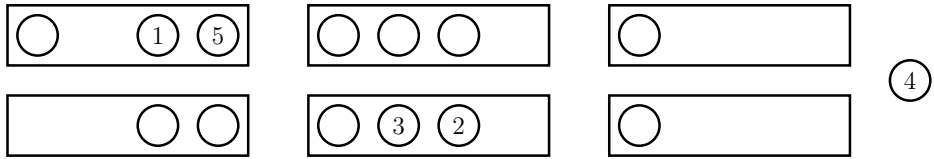
\begin{center}
\scalebox{0.8}{
\begin{tikzpicture}[main/.style = {draw, circle, very thick}] 
\draw [very thick] (0,0) rectangle (4,1); \draw [very thick] (5,0) rectangle (9,1); \draw [very thick] (10,0) rectangle (14,1);
\draw [very thick] (0,1.5) rectangle (4,2.5); \draw [very thick] (5,1.5) rectangle (9,2.5); \draw [very thick] (10,1.5) rectangle (14,2.5);
\node[main](0) at (0.5,2) {{\color{white}\large $0$}};  \node[main](1) at (2.5,2) {{\large $1$}};
\node[main](2) at (3.5,2) {{\large $5$}};
\node[main](3) at (5.5,2) {{\color{white}\large $0$}}; \node[main](3) at (6.5,2) {{\color{white}\large $0$}};
\node[main](5) at (7.5,2) {{\color{white}\large $0$}};\node[main](6) at (10.5,2) {{\color{white}\large $0$}}; 
\node[main](7) at (15,1.25) {{\large $4$}};
\node[main](8) at (10.5,0.5) {{\color{white}\large $0$}};
\node[main](9) at (7.5,0.5) {{\large $2$}};\node[main](10) at (6.5,0.5) {{\large $3$}};
\node[main](11) at (5.5,0.5) {{\color{white}\large $0$}};
\node[main](12) at (3.5,0.5) {{\color{white}\large $0$}};\node[main](13) at (2.5,0.5) {{\color{white}\large $0$}};
\end{tikzpicture}
}
\end{center}
Combining (i) and (i\hspace{-0.5mm}i), we have $|M(\Cc^{(1)}_q)|=2(q-2m)^{m-i}(q-2m+1)^{i-1}$ if $q$ is even.

\subsection{Characteristic quasi-polynomial of restriction on $\{2x_i=1\}$}\label{sec-TypeC-rest-2x_i=1}
Let $\Cc^{(2)}=\mathcal{C}_m^{\{2x_i=1\}}$ for some $i\in [m]$, and we prove the second equality of Theorem \ref{thm-main-typeC} as follows.
\begin{Them}\label{thm-char-quasi-TypeC-2xi=1}
We have
\begin{align*}
\left|M(\Cc_q^{(2)})\right|=
\begin{cases}
(q-2m)^{i-1}(q-2m+1)^{m-i}\;\;\;&\text{if $q$ is odd},\\
0 &\text{if $q$ is even}
\end{cases}
\end{align*}
for any $q\in\mathbb{Z}$ with $q\gg 0$.
\end{Them}
The complement $M(\Cc^{(2)}_q)$ is the set of $(x_1,\dots,x_m)\in\mathbb{Z}_q^m$ that satisfies the following conditions:
\begin{align*}
&2x_i=1,\\
&2x_s\neq 0,\ 2x_s\neq 1\ (s\in [m],\ s\neq i),\\
&x_s\neq x_t\ (s,t\in [m],\ s\neq t),\\
&x_s\neq x_t+1\ (s,t\in [m],\ s<t),\\
&x_s\neq -x_t,\ x_s\neq -x_t+1\ (s,t\in [m],\ s\neq t).
\end{align*}
We fix the condition $2x_i=1$ and count the number of elements in $(x_1,\dots,x_m)\in M(\Cc^{(2)}_q)$ similarly to the previous subsection. 
If $q$ is even, since $\{\bm{x}\in\mathbb{Z}_q^m\mid 2x_i=1\} = \emptyset$, we immediately obtain $\left|M(\Cc^{(2)})\right|=0$.

\medskip
In what follows, let $q$ be odd. 
Since $\displaystyle 2x_i=1\ \Leftrightarrow\ x_i=\frac{q+1}{2}$, the condition $2x_i=1$ means that the circle at the right edge of the lower right box has the label $i$.

Prepare $q-2m+1$ boxes side by side, $\displaystyle \frac{q+1}{2}-m$ on the upper side and $\displaystyle \frac{q+1}{2}-m$ on the lower side.
Place the circle with the label $i$ at the right edge of the lower right box and the unlabeled circles on the opposite side of the circle with the label $i$.
Then, according to the following procedure, place each of the numbers $1,\dots,i-1,i+1,\dots,m$ in one of the $q-2m+1$ boxes.
\begin{itemize}
\item Each of the numbers $1,\dots,i-1$ cannot be placed in the lower right box, since $x_s\neq x_i+1$ for $s<i$.
There are $q-2m$ boxes that can contain the numbers $1,\dots,i-1$, except the lower right boxes.

\item Since the numbers $i+1,\dots,m$ can be placed in all boxes, there are $q-2m+1$ boxes that can contain the numbers $i+1,\dots,m$.
\end{itemize}
Therefore, we have $|M(\Cc^{(2)}_q)|=(q-2m)^{i-1}(q-2m+1)^{m-i}$ in this case.
For example, the following boxes and circles correspond to the element $(x_1,x_2,x_3,x_4,x_5)=(1,3,4,8,9)\in M(\Cc^{(2)}_q)$ in the case of $m=5,q=15,i=4$:
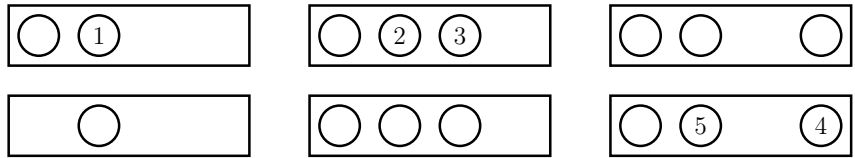
\begin{center}
\scalebox{0.8}{
\begin{tikzpicture}[main/.style = {draw, circle, very thick}] 
\draw [very thick] (0,0) rectangle (4,1); \draw [very thick] (5,0) rectangle (9,1); \draw [very thick] (10,0) rectangle (14,1);
\draw [very thick] (0,1.5) rectangle (4,2.5); \draw [very thick] (5,1.5) rectangle (9,2.5); \draw [very thick] (10,1.5) rectangle (14,2.5);
\node[main](0) at (0.5,2) {{\color{white}\large $0$}};  \node[main](1) at (1.5,2) {{\large $1$}};
\node[main](2) at (5.5,2) {{\color{white}\large $0$}}; \node[main](3) at (6.5,2) {{\large $2$}};
\node[main](4) at (7.5,2) {{\large $3$}}; \node[main](5) at (10.5,2) {{\color{white}\large $0$}};\node[main](6) at (11.5,2) {{\color{white}\large $0$}}; 
\node[main](7) at (13.5,2) {{\color{white}\large $0$}};\node[main](8) at (13.5,0.5) {{\large $4$}};
\node[main](9) at (11.5,0.5) {{\large $5$}}; \node[main](10) at (10.5,0.5) {{\color{white}\large $0$}};
\node[main](11) at (7.5,0.5) {{\color{white}\large $0$}}; \node[main](12) at (6.5,0.5) {{\color{white}\large $0$}};
\node[main](13) at (5.5,0.5) {{\color{white}\large $0$}};\node[main](14) at (1.5,0.5) {{\color{white}\large $0$}};
\end{tikzpicture}
}
\end{center}

\subsection{Characteristic quasi-polynomial of restriction on $\{x_i-x_j=0\}$}\label{sec-TypeC-rest-xi=xj}
Let $\Cc^{(3)}=\mathcal{C}_m^{\{x_i-x_j=0\}}$ for $1\leq i<j \leq m$, and we prove the third equality of Theorem \ref{thm-main-typeC} as follows.
\begin{Them}\label{thm-char-quasi-TypeC-xi=xj}
We have
\begin{align*}
\left|M(\Cc_q^{(3)})\right|=(q-2m+1)^{j-i-1}(q-2m+2)^{m-j+i}
\end{align*}
for any $q\in\mathbb{Z}$ with $q\gg 0$.
\end{Them}
The complement $M(\Cc^{(3)}_q)$ is the set of $(x_1,\dots,x_m)\in\mathbb{Z}_q^m$ that satisfies the following conditions:
\begin{align*}
&x_i=x_j,\\
&2x_s\neq 0,\ 2x_s\neq 1\ (s\in [m]),\\
&x_s\neq x_t\ (s,t\in [m],\ s\neq t,\ \{s,t\}\neq \{i,j\}),\\
&x_s\neq x_t+1\ (s,t\in [m],\ s<t,\ \{s,t\}\neq \{i,j\}),\\
&x_s\neq -x_t,\ x_s\neq -x_t+1\ (s,t\in [m],\ s\neq t).
\end{align*}
We fix the condition $x_i=x_j$ and count the number of elements in $(x_1,\dots,x_m)\in M(\Cc^{(3)}_q)$. 
The condition $x_i=x_j$ means that a circle has the label $i$ and the label $j$ at the same time.
We may consider the number $i$ as the pair $(i,j)$ in this case.

\subsubsection{The case where $q$ is odd}\label{sec-TypeC-rest-x_i=x_j-odd}
Let $q$ be odd.
Prepare $q-2m+3$ boxes side by side, $\displaystyle \frac{q+1}{2}-(m-1)$ on the upper side and $\displaystyle \frac{q+1}{2}-(m-1)$ on the lower side.
Place each of the numbers $1,\dots,j-1,j+1,\dots,m$ in one of the $q-2m+3$ boxes and create boxes and circles corresponding to the element $(x_1,\dots,x_m)\in M(\Cc^{(3)}_q)$.
From the condition $2x_s\neq 1$ for $s\in[m]$, the number $s$ cannot be placed in the lower right box.
\begin{itemize}
\item We first place the number $i$.
There are $q-2m+2$ boxes that can contain the number $i$, except the lower right box.
\item Next, we place the numbers $1,\dots,i-1$. Let $s<i$.
If the number $s$ is placed in the same box as the number $i$, then the circle with the label $s$ cannot be placed immediately after the circle with the label $i$ in the clockwise direction, since $x_s\neq x_i+1$ for $s<i$.
However, the circle with the label $s$ can be placed immediately before the circle with the label $i$ in the clockwise direction.
Therefore, the number $s$ can be placed in the same box as the number $i$.
There are $q-2m+2$ boxes that can contain the numbers $1,\dots,i-1$, except the lower right box.
\item Then, we place the numbers $i+1,\dots,j-1$. Let $i<s<j$.
The circle with the label $s$ cannot be placed immediately before and immediately after the circle with the label $i$ in the clockwise direction, since $x_s\neq x_i+1$ and $x_j\neq x_s+1$ for $i<s<j$.
Therefore, the number $s$ cannot be placed in the same box as the number $i$.
There are $q-2m+1$ boxes that can contain the numbers $i+1,\dots,j-1$, except the lower right box and the box that contains the number $i$.
\item Finally, we place the numbers $j+1,\dots,m$. Let $j<s$.
If the number $s$ is placed in the same box as the number $i$, then the circle with the label $s$ is placed immediately after the circle with the label $i$ in the clockwise direction.
The number $s$ can be placed in the same box as the number $i$.
There are $q-2m+2$ boxes that can contain the numbers $j+1,\dots,m$, except the lower right box.
\end{itemize}
Therefore, we have $|M(\Cc^{(3)}_q)|=(q-2m+1)^{j-i-1}(q-2m+2)^{m-j+i}$ in this case.
For example, the following boxes and circles correspond to the element $(x_1,x_2,x_3,x_4,x_5)=(5,3,1,3,6)\in M(\Cc^{(3)}_q)$ in the case of $m=5,q=13,i=2,j=4$:
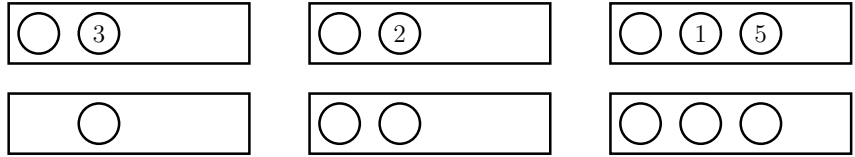
\begin{center}
\scalebox{0.8}{
\begin{tikzpicture}[main/.style = {draw, circle, very thick}] 
\draw [very thick] (0,0) rectangle (4,1); \draw [very thick] (5,0) rectangle (9,1); \draw [very thick] (10,0) rectangle (14,1);
\draw [very thick] (0,1.5) rectangle (4,2.5); \draw [very thick] (5,1.5) rectangle (9,2.5); \draw [very thick] (10,1.5) rectangle (14,2.5);
\node[main](0) at (0.5,2) {{\color{white}\large $0$}};  \node[main](3) at (12.5,2) {{\large $5$}};
\node[main](4) at (5.5,2) {{\color{white}\large $0$}}; \node[main](5) at (6.5,2) {{\large $2$}};
\node[main](6) at (1.5,2) {{\large $3$}}; \node[main](7) at (10.5,2) {{\color{white}\large $0$}};
\node[main](1) at (11.5,2) {{\large $1$}};  \node[main](2) at (11.5,0.5) {{\color{white}\large $0$}};
\node[main](12) at (12.5,0.5) {{\color{white}\large $0$}};
\node[main](11) at (5.5,0.5) {{\color{white}\large $0$}}; \node[main](10) at (6.5,0.5) {{\color{white}\large $0$}};
\node[main](9) at (10.5,0.5) {{\color{white}\large $0$}}; \node[main](8) at (1.5,0.5) {{\color{white}\large $0$}};
\end{tikzpicture}
}
\end{center}

\subsubsection{The case where $q$ is even}\label{sec-TypeC-rest-x_i=x_j-even}
Let $q$ be even.
Prepare $q-2m+2$ boxes side by side, $\displaystyle \frac{q}{2}-(m-1)$ on the upper side and $\displaystyle \frac{q}{2}-(m-1)$ on the lower side, and a circle corresponding to the element $\displaystyle \frac{q}{2}\in\mathbb{Z}_q$ on the right side of the boxes.
Place each of the numbers $1,\dots,j-1,j+1,\dots,m$ in one of the $q-2m+2$ boxes.
Unlike the case where $q$ is odd, there does not exist $x_s\in\mathbb{Z}_q$ such that $2x_s=1$, so each of the numbers $1,\dots,j-1,j+1,\dots,m$ can be placed in the lower right box.
The circle on the right side of the boxes has no labels and the number of boxes is one less than when $q$ is odd. 
In addition, the other prohibitive conditions are exactly the same as when $q$ is odd.
Therefore, we also have $|M(\Cc^{(3)}_q)|=(q-2m+2)^{m-j+i}(q-2m+1)^{j-i-1}$ if $q$ is even.

\subsection{Characteristic quasi-polynomial of restriction on $\{x_i-x_j=1\}$}\label{sec-TypeC-rest-xi=xj+1}
Let $\Cc^{(4)}=\mathcal{C}_m^{\{x_i-x_j=1\}}$ for $1\leq i<j \leq m$, and we prove the fourth equality of Theorem \ref{thm-main-typeC} as follows.
\begin{Them}\label{thm-char-quasi-TypeC-xi=xj+1}
We have
\begin{align*}
\left|M(\Cc_q^{(4)})\right|=(q-2m)^{m-j+i}(q-2m+1)^{j-i-1}
\end{align*}
for any $q\in\mathbb{Z}$ with $q\gg 0$.
\end{Them}
The complement $M(\Cc^{(4)}_q)$ is the set of $(x_1,\dots,x_m)\in\mathbb{Z}_q^m$ that satisfies the following conditions:
\begin{align*}
&x_j=x_i-1,\\
&2x_s\neq 0,\ 2x_s\neq 1\ (s\in [m]),\\
&x_s\neq x_t\ (s,t\in [m],\ s\neq t),\\
&x_s\neq x_t+1\ (s,t\in [m],\ s<t,\ \{s,t\}\neq \{i,j\}),\\
&x_s\neq -x_t,\ x_s\neq -x_t+1\ (s,t\in [m],\ s\neq t).
\end{align*}
We fix the condition $x_j=x_i-1$ and count the number of elements in $(x_1,\dots,x_m)\in M(\Cc^{(4)}_q)$. 
The condition $x_j=x_i-1$ means that the circle with the label $j$ is placed immediately before the circle with the label $i$ in the clockwise direction.
Therefore, once the box containing the number $i$ is determined, there is no need to specify the box containing the number $j$.

\subsubsection{The case where $q$ is odd}\label{sec-TypeC-rest-x_i=x_j+1-odd}
Let $q$ be odd.
Prepare $q-2m+1$ boxes side by side, $\displaystyle \frac{q+1}{2}-m$ on the upper side and $\displaystyle \frac{q+1}{2}-m$ on the lower side.
Place each of the numbers $1,\dots,j-1,j+1,\dots,m$ in one of the $q-2m+1$ boxes and create boxes and circles corresponding to the element $(x_1,\dots,x_m)\in M(\Cc^{(4)}_q)$.
From the condition $2x_s\neq 1$ for $s\in[m]$, the number $s$ cannot be placed in the lower right box.
\begin{itemize}
\item We first place the number $i$.
There are $q-2m$ boxes that can contain the number $i$, except the lower right box.
\item Next, we place the numbers $1,\dots,i-1$. Let $s<i$.
If the number $s$ is placed in the same box as the number $i$, then the circle with the label $s$ is placed before the circle labeled with $j$ in the clockwise direction.
In this case, we have $x_j=x_s+1$ and $s<j$, but no contradiction occurs.
However, the circle with the label $s$ cannot be placed after the circle with the label $i$, since $x_s\neq x_i+1$ with $s<i$.
The number $s$ can be placed in the same box as the number $i$ in one way.
There are $q-2m$ boxes that can contain the numbers $1,\dots,i-1$, except the lower right box.
\item Then, we place the numbers $i+1,\dots,j-1$. Let $i<s<j$.
When placing the number $s$ in the same box that contains the number $i$, we have two choices before $j$ or after $i$ in the clockwise direction.
If the circle with the label $s$ is before $j$, then we have $x_j=x_s+1$ and $s<j$ and there is no contradiction.
If the circle with the label $s$ is after $i$, then we have $x_s=x_i+1$ and $i<s$ and there is also no contradiction.
Therefore, there are $q-2m+1$ choices, since we have $q-2m$ boxes except the lower right box, and there are two choices for the box that contains the number $i$.
\item Finally, we place the numbers $j+1,\dots,m$. Let $j<s$.
In this case, if the number $s$ is placed in the same box as the number $i$, then the circle with the label $s$ is placed after the circle labeled with $i$ in the clockwise direction.
There are $q-2m$ boxes that can contain the numbers $j+1,\dots,m$, except the lower right box.
\end{itemize}
Therefore, we have $|M(\Cc^{(4)}_q)|=(q-2m)^{m-j+i}(q-2m+1)^{j-i-1}$ in this case.
For example, the following boxes and circles correspond to the element $(x_1,x_2,x_3,x_4,x_5)=(1,4,5,3,7)\in M(\Cc^{(4)}_q)$ in the case of $m=5,q=15,i=2,j=4$:
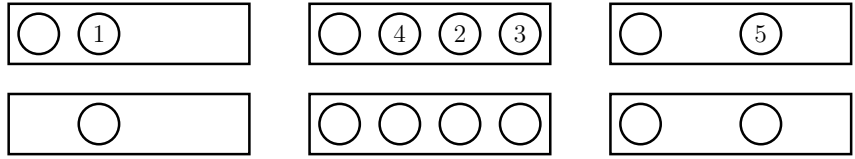
\begin{center}
\scalebox{0.8}{
\begin{tikzpicture}[main/.style = {draw, circle, very thick}] 
\draw [very thick] (0,0) rectangle (4,1); \draw [very thick] (5,0) rectangle (9,1); \draw [very thick] (10,0) rectangle (14,1);
\draw [very thick] (0,1.5) rectangle (4,2.5); \draw [very thick] (5,1.5) rectangle (9,2.5); \draw [very thick] (10,1.5) rectangle (14,2.5);
\node[main](0) at (0.5,2) {{\color{white}\large $0$}};  \node[main](3) at (12.5,2) {{\large $5$}};
\node[main](4) at (5.5,2) {{\color{white}\large $0$}}; \node[main](5) at (6.5,2) {{\large $4$}};
\node[main](5) at (8.5,2) {{\large $3$}};\node[main](4) at (8.5,0.5) {{\color{white}\large $0$}};
\node[main](5) at (7.5,2) {{\large $2$}};\node[main](4) at (7.5,0.5) {{\color{white}\large $0$}};
\node[main](6) at (1.5,2) {{\large $1$}}; \node[main](7) at (10.5,2) {{\color{white}\large $0$}};
\node[main](12) at (12.5,0.5) {{\color{white}\large $0$}};
\node[main](11) at (5.5,0.5) {{\color{white}\large $0$}}; \node[main](10) at (6.5,0.5) {{\color{white}\large $0$}};
\node[main](9) at (10.5,0.5) {{\color{white}\large $0$}}; \node[main](8) at (1.5,0.5) {{\color{white}\large $0$}};
\end{tikzpicture}
}
\end{center}

\subsubsection{The case where $q$ is even}\label{sec-TypeC-rest-x_i=x_j+1-even}
Let $q$ be even.
Exactly as in Section \ref{sec-TypeC-rest-x_i=x_j-even}, the number of elements of $M(\Cc^{(4)}_q)$ can be calculated as in the same discussion when $q$ is odd.
Therefore, we also have $|M(\Cc^{(4)}_q)|=(q-2m)^{m-j+i}(q-2m+1)^{j-i-1}$ in this case.

\subsection{Characteristic quasi-polynomial of restriction on $\{x_i+x_j=0\}$}\label{sec-TypeC-rest-xi=-xj}
Let $\Cc^{(5)}=\mathcal{C}_m^{\{x_i+x_j=0\}}$ for $1\leq i<j \leq m$, and we prove the fifth equality of Theorem \ref{thm-main-typeC} as follows.
\begin{Them}\label{thm-char-quasi-TypeC-xi=-xj}
We have
\begin{align*}
\left|M(\Cc_q^{(5)})\right|=
\begin{cases}
(q-2m)^{m-j}(q-2m+1)^{j-i}(q-2m+2)^{i-1}&\text{if $q$ is odd},\\
(q-2m)^{m-j}(q-2m+1)^{j-i-1}(q-2m+2)^{i}&\text{if $q$ is even}
\end{cases}
\end{align*}
for any $q\in\mathbb{Z}$ with $q\gg 0$.
\end{Them}
The complement $M(\Cc^{(5)}_q)$ is the set of $(x_1,\dots,x_m)\in\mathbb{Z}_q^m$ that satisfies the following conditions:
\begin{align*}
&x_j=-x_i,\\
&2x_s\neq 0,\ 2x_s\neq 1\ (s\in [m]),\\
&x_s\neq x_t\ (s,t\in [m],\ s\neq t),\\
&x_s\neq x_t+1\ (s,t\in [m],\ s<t),\\
&x_s\neq -x_t,\ x_s\neq -x_t+1\ (s,t\in [m],\ s\neq t,\ \{s,t\}\neq \{i,j\}).
\end{align*}
We fix the condition $x_j=-x_i$ and count the number of elements in $(x_1,\dots,x_m)\in M(\Cc^{(5)}_q)$. 
The condition $x_j=-x_i$ means that the circle with the label $j$ is placed just on the opposite side of the circle labeled with $i$. 
Therefore, once the box containing the number $i$ is determined, there is no need to specify the box containing the number $j$.

\subsubsection{The case where $q$ is odd}\label{sec-TypeC-rest-x_i=-x_j-odd}
Let $q$ be odd. 
Prepare $q-2m+3$ boxes side by side, $\displaystyle \frac{q+1}{2}-(m-1)$ on the upper side and $\displaystyle \frac{q+1}{2}-(m-1)$ on the lower side.
Place each of the numbers $1,\dots,j-1,j+1,\dots,m$ in one of the $q-2m+3$ boxes.
For $s\in[m]$, the number $s$ cannot be placed in the lower right box since $2x_s\neq 1$.
If the circle with the label $k$ is placed in the same box as the circle with the label $i$ (resp. $j$) immediately after the circle with the label $i$ (resp. $j$) in the clockwise direction, then we have $-x_k+1=x_j$ (resp. $-x_k+1=x_i$), which is a contradiction. In other words, if a labeled circle is put in the same box as the circle with the label $i$ (resp. $j$), it can only be before the circle with the label $i$ (resp. $j$).
\begin{itemize}
\item We first place the number $i$.
The number $i$ cannot be placed in the upper right box.
Indeed, if the number $i$ is in the upper right box, then the circle with the label $i$ is at the right edge of the upper right box.
In addition, the circle with the label $j$ is at the right edge of the lower right box, which contradicts the condition $2x_j\neq 1$.
Therefore, there are $q-2m+1$ boxes that can contain the number $i$, except the upper right and lower right boxes.
\item Next, we place the numbers $1,\dots,i-1$. Let $s<i$.
If the number $s$ is placed in the same box as the number $i$ (resp. $j$), then the circle with the label $s$ is placed before the circle with the label $i$ (resp. $j$) in the clockwise direction.
There are $q-2m+2$ boxes that can contain the numbers $1,\dots,i-1$, except the lower right box.
\item Then, we place the numbers $i+1,\dots,j-1$. Let $i<s<j$.
If the number $s$ is placed in the same box as the number $i$, then the circle with the label $s$ must be placed after the circle with the label $i$ in the clockwise direction.
This is a contradiction.
On the other hand, the number $s$ can be placed in the same box as the number $j$.
Therefore, there are $q-2m+1$ boxes that can contain the numbers $i+1,\dots,j-1$, except the lower right box and the box that contains the number $i$.
\item Finally, we place the numbers $j+1,\dots,m$. Let $j<s$.
The number $s$ cannot be placed in the two boxes that contain the numbers $i$ and $j$.
There are $q-2m$ boxes that can contain the numbers $j+1,\dots,m$, except the lower right box and the boxes that contain the numbers $i$ and $j$.
\end{itemize}
Therefore, we have $|M(\Cc^{(5)}_q)|=(q-2m)^{m-j}(q-2m+1)^{j-i}(q-2m+2)^{i-1}$ in this case.

\subsubsection{The case where $q$ is even}\label{sec-TypeC-rest-x_i=-x_j-even}
Let $q$ be even.
Prepare $q-2m+2$ boxes side by side, $\displaystyle \frac{q}{2}-(m-1)$ on the upper side and $\displaystyle \frac{q}{2}-(m-1)$ on the lower side, and a circle corresponding to the element $\displaystyle \frac{q}{2}\in\mathbb{Z}_q$ on the right side of the boxes.
Place each of the numbers $1,\dots,j-1,j+1,\dots,m$ in one of the $q-2m+2$ boxes.
Unlike the case where $q$ is odd, each of the numbers $1,\dots,j-1,j+1,\dots,m$ can be placed in the lower right box.
The circle on the right side of the boxes has no labels since $2x_s\neq 0$ for $s\in [m]$.
\begin{itemize}
\item The number $i$ can be placed in all boxes.
Therefore, there are $q-2m+2$ boxes that can contain the number $i$.
\item For the numbers $1,\dots,i-1,i+1,\dots,j-1,j+1,\dots,m$, exactly as in the case where $q$ is odd, we can determine boxes that cannot contain each number.
There are $q-2m+2$ boxes that can contain the numbers $1,\dots,i-1$, since the numbers $1,\dots,i-1$ can be placed in all boxes.
Next, there are $q-2m+1$ boxes that can contain the numbers $i+1,\dots,j-1$, except the box that contains the number $i$.
Finally, there are $q-2m$ boxes that can contain the numbers $j+1,\dots,m$, except the boxes that contain the numbers $i$ and $j$.
\end{itemize}
Therefore, we have $|M(\Cc^{(5)}_q)|=(q-2m)^{m-j}(q-2m+1)^{j-i-1}(q-2m+2)^{i}$ in this case.
For example, the following boxes and circles correspond to the element $(x_1,x_2,x_3,x_4,x_5)=(8,3,1,11,9)\in M(\Cc^{(5)}_q)$ in the case of $m=5,q=14,i=2,j=4$:
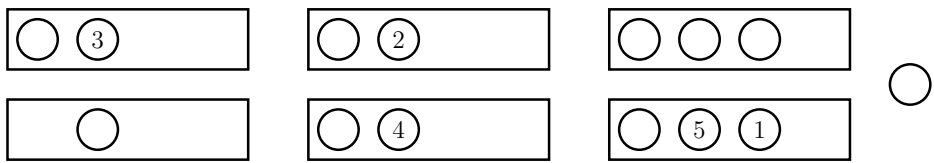
\begin{center}
\scalebox{0.8}{
\begin{tikzpicture}[main/.style = {draw, circle, very thick}] 
\draw [very thick] (0,0) rectangle (4,1); \draw [very thick] (5,0) rectangle (9,1); \draw [very thick] (10,0) rectangle (14,1);
\draw [very thick] (0,1.5) rectangle (4,2.5); \draw [very thick] (5,1.5) rectangle (9,2.5); \draw [very thick] (10,1.5) rectangle (14,2.5);
\node[main](0) at (0.5,2) {{\color{white}\large $0$}};  \node[main](3) at (12.5,2) {\color{white}\large $0$};
\node[main](4) at (5.5,2) {{\color{white}\large $0$}}; \node[main](5) at (6.5,2) {{\large $2$}};
\node[main](6) at (1.5,2) {{\large $3$}}; \node[main](7) at (10.5,2) {{\color{white}\large $0$}};
\node[main](1) at (11.5,2) {{\color{white}\large $0$}};  \node[main](2) at (11.5,0.5) {{\large $5$}};
\node[main](12) at (12.5,0.5) {{\large $1$}};\node[main](0) at (15,1.25) {{\color{white}\large $0$}};
\node[main](11) at (5.5,0.5) {{\color{white}\large $0$}}; \node[main](10) at (6.5,0.5) {{\large $4$}};
\node[main](9) at (10.5,0.5) {{\color{white}\large $0$}}; \node[main](8) at (1.5,0.5) {{\color{white}\large $0$}};
\end{tikzpicture}
}
\end{center}

\subsection{Characteristic quasi-polynomial of restriction on $\{x_i+x_j=1\}$}\label{sec-TypeC-rest-xi=-xj+1}
Let $\Cc^{(6)}=\mathcal{C}_m^{\{x_i+x_j=1\}}$ for $1\leq i<j \leq m$, and we prove the sixth equality of Theorem \ref{thm-main-typeC} as follows.
\begin{Them}\label{thm-char-quasi-TypeC-xi=-xj+1}
We have
\begin{align*}
\left|M(\Cc_q^{(6)})\right|=
\begin{cases}
(q-2m)^{i-1}(q-2m+1)^{j-i}(q-2m+2)^{m-j}&\text{if $q$ is odd},\\
(q-2m)^{i}(q-2m+1)^{j-i-1}(q-2m+2)^{m-j}&\text{if $q$ is even}
\end{cases} 
\end{align*}
for any $q\in\mathbb{Z}$ with $q\gg 0$.
\end{Them}
The complement $M(\Cc^{(6)}_q)$ is the set of $(x_1,\dots,x_m)\in\mathbb{Z}_q^m$ that satisfies the following conditions:
\begin{align*}
&x_j=-x_i+1,\\
&2x_s\neq 0,\ 2x_s\neq 1\ (s\in [m]),\\
&x_s\neq x_t\ (s,t\in [m],\ s\neq t),\\
&x_s\neq x_t+1\ (s,t\in [m],\ s<t),\\
&x_s\neq -x_t\ (s,t\in[m],\ s\neq t),\\
&x_s\neq -x_t+1\
(s,t\in[m],\ s\neq t,\ \{s,t\}\neq\{i,j\}).
\end{align*}
We fix the condition $x_j=-x_i+1$ and count the number of elements in $(x_1,\dots,x_m)\in M(\Cc^{(6)}_q)$. 
The condition $x_j=-x_i+1$ means that the circle with label $j$ is immediately after the opposite circle of the circle with label $i$.
Therefore, once the box containing the number $i$ is determined, there is no need to specify the box containing the number $j$.

\subsubsection{The case where $q$ is odd}\label{sec-TypeC-rest-x_i=-x_j+1-odd}
Let $q$ be odd.
Prepare $q-2m+1$ boxes side by side, $\displaystyle \frac{q+1}{2}-m$ on the upper side and $\displaystyle \frac{q+1}{2}-m$ on the lower side.
Place each of the numbers $1,\dots,j-1,j+1,\dots,m$ in one of the $q-2m+1$ boxes.

We first place the number $i$.
The number $i$ can be placed in the lower right box.
In fact, if $i$ is placed in the lower right box, then the circle immediately before the circle with the label $i$ is the unlabeled circle in the clockwise direction, and we have $\displaystyle 2x_i\neq 1\ \Leftrightarrow x_i\neq \frac{q+1}{2}$.
Similarly, the number $i$ can be placed in the upper right box.
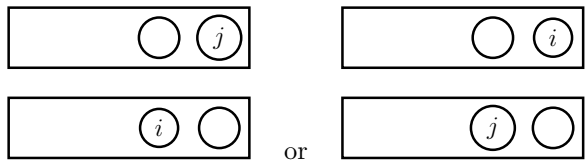
\begin{center}
\scalebox{0.8}{
\begin{tikzpicture}[main/.style = {draw, circle, very thick}] 
\draw [very thick] (0,0) rectangle (4,1); \draw [very thick] (0,1.5) rectangle (4,2.5);
\node[main](1) at (2.5,2) {{\color{white}\large $0$}}; \node[main](2) at (3.5,2) {{\large $j$}}; \node[main](2) at (2.5,0.5) {{\large $i$}};\node[main](1) at (3.5,0.5) {{\color{white}\large $0$}};
\end{tikzpicture}
}\quad or\quad
\scalebox{0.8}{
\begin{tikzpicture}[main/.style = {draw, circle, very thick}] 
\draw [very thick] (0,0) rectangle (4,1); \draw [very thick] (0,1.5) rectangle (4,2.5);
\node[main](1) at (2.5,2) {{\color{white}\large $0$}}; \node[main](2) at (3.5,2) {{\large $i$}}; \node[main](2) at (2.5,0.5) {{\large $j$}};\node[main](1) at (3.5,0.5) {{\color{white}\large $0$}};
\end{tikzpicture}
}
\end{center}
In the following, the discussion will be divided into cases, depending on where the number $i$ is placed.

(i) We consider the case where $i$ is placed in the lower right box.
Let $s<i$.
The number $s$ cannot be placed in the lower right box, as the circle with the label $s$ is placed before the circle with the label $i$ in the clockwise direction, a contradiction to the condition $2x_s\neq 1$.
There are $q-2m$ boxes that can contain the numbers $1,\dots,i-1$, except the lower right box.
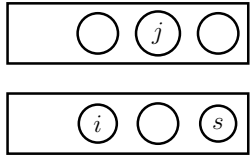
\begin{center}
\scalebox{0.8}{
\begin{tikzpicture}[main/.style = {draw, circle, very thick}] 
\draw [very thick] (0,0) rectangle (4,1); \draw [very thick] (0,1.5) rectangle (4,2.5);
\node[main](1) at (1.5,2) {{\color{white}\large $0$}}; \node[main](2) at (2.5,2) {{\large $j$}}; \node[main](3) at (3.5,2) {{\color{white}\large $0$}}; \node[main](4) at (1.5,0.5) {{\large $i$}}; \node[main](5) at (2.5,0.5) {{\color{white}\large $0$}}; \node[main](6) at (3.5,0.5) {{\large $s$}};
\end{tikzpicture}
}
\end{center}
Next, let $i<s<j$.
The number $s$ can be placed in the lower right box, as the circle with the label $s$ is placed after the circle with the label $i$ in the clockwise direction.
Therefore, there are $q-2m+1$ boxes that can contain the numbers $i+1,\dots,j-1$, since these numbers can be placed in all boxes.

Finally, let $j<s$.
The number $s$ can be placed in the lower right box.
In addition, there are two ways to place the circle with the label $s$ in the upper right box, one is to place $s$ before $j$ and the other is to place $s$ after $j$.
There are $q-2m+2$ choices for the numbers $j+1,\dots,m$, since we have $q-2m+1$ boxes and there are two choices for the upper right box.

The number of elements of $M(\Cc^{(6)}_q)$ in this case is $(q-2m)^{i-1}(q-2m+1)^{j-i-1}(q-2m+2)^{m-j}$.
For example, the following boxes and circles correspond to the element $(x_1,x_2,x_3,x_4,x_5,x_6)=(3,11,12,7,8,13)\in M(\Cc^{(6)}_q)$ in the case of $m=6,q=17,i=2,j=4$:
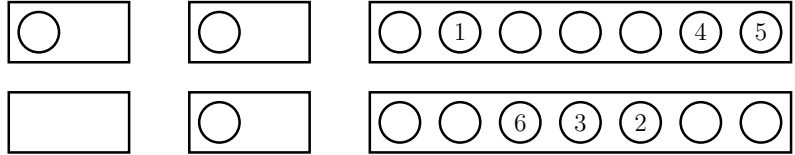
\begin{center}
\scalebox{0.8}{
\begin{tikzpicture}[main/.style = {draw, circle, very thick}] 
\draw [very thick] (0,0) rectangle (2,1); \draw [very thick] (3,0) rectangle (5,1); \draw [very thick] (6,0) rectangle (13,1);
\draw [very thick] (0,1.5) rectangle (2,2.5); \draw [very thick] (3,1.5) rectangle (5,2.5); \draw [very thick] (6,1.5) rectangle (13,2.5);
\node[main](0) at (0.5,2) {{\color{white}\large $0$}}; \node[main](1) at (3.5,2) {{\color{white}\large $0$}}; \node[main](2) at (6.5,2) {{\color{white}\large $0$}}; \node[main](3) at (7.5,2) {{\large $1$}}; \node[main](4) at (8.5,2) {{\color{white}\large $0$}}; \node[main](5) at (9.5,2) {{\color{white}\large $0$}}; \node[main](6) at (10.5,2) {{\color{white}\large $0$}}; \node[main](7) at (11.5,2) {{\large $4$}}; \node[main](8) at (12.5,2) {{\large $5$}};
\node[main](9) at (12.5,0.5) {{\color{white}\large $0$}}; \node[main](10) at (11.5,0.5) {{\color{white}\large $0$}}; \node[main](11) at (10.5,0.5) {{\large $2$}}; \node[main](12) at (9.5,0.5) {{\large $3$}}; \node[main](13) at (8.5,0.5) {{\large $6$}}; \node[main](14) at (7.5,0.5) {{\color{white}\large $0$}}; \node[main](15) at (6.5,0.5) {{\color{white}\large $0$}}; \node[main](16) at (3.5,0.5) {{\color{white}\large $0$}};
\end{tikzpicture}
}
\end{center}

(i\hspace{-0.5mm}i) We consider the case where $i$ is placed in the upper right box.
Let $s<i$.
The number $s$ cannot be placed in the lower right box, as the circle with the label $s$ is placed before the circle with the label $j$ in the clockwise direction.
There are $q-2m$ boxes that can contain the numbers $1,\dots,i-1$, except the lower right box.

Next, let $i<s<j$.
The number $s$ cannot also be placed in the lower right box, and there are two ways (before $i$ and after $i$) to place the circle with the label $s$ in the upper right box.
There are $q-2m+1$ choices for the numbers $i+1,\dots,j-1$, since we have $q-2m$ boxes except the lower right box and there are two choices for the upper right box.

Finally, let $j<s$.
The number $s$ can be placed in the lower right box, and there are two ways (before $i$ and after $i$) to place the circle with the label $s$ in the upper right box.
Therefore, there are $q-2m+2$ choices for the numbers $j+1,\dots,m$, since we have $q-2m+1$ boxes and there are two choices for the upper right box.

The number of elements of $M(\Cc^{(6)}_q)$ in this case is also $(q-2m)^{i-1}(q-2m+1)^{j-i-1}(q-2m+2)^{m-j}$.

(i\hspace{-0.5mm}i\hspace{-0.5mm}i) We consider the case where $i$ is placed in neither the upper right box nor the lower right box. 
There are $q-2m-1$ boxes that can contain the number $i$, except the upper right and lower right boxes.
Then, each of the numbers $1,\dots,i-1$ cannot be placed in the lower right box.
There are $q-2m$ boxes that can contain the numbers $1,\dots,i-1$, except the lower right box.

Next, let $i<s<j$.
The number $s$ cannot be placed in the lower right box.
In addition, there are two ways (before $i$ and after $i$) to place the circle with the label $s$ in the same box as the number $i$, but the circle with the label $s$ can only be placed before $j$ in the same box as the number $j$.
In fact, if $s$ is placed after $j$, then we have $x_s=x_j+1$ with $s<j$, a contradiction.
Therefore, there are $q-2m+1$ choices for the numbers $i+1,\dots,j-1$, since we have $q-2m$ boxes except the lower right box and there are two choices for the same box as $i$.

Finally, let $j<s$.
The number $s$ cannot be placed in the lower right box.
There are two ways to place the circle with the label $s$ in the same boxes as the numbers $i$ and $j$.
Therefore, there are $q-2m+2$ choices for the numbers $j+1,\dots,m$, since we have $q-2m$ boxes except the lower right box and there are two choices for the same boxes as $i$ and $j$.

The number of elements of $M(\Cc^{(6)}_q)$ in this case is $(q-2m-1)(q-2m)^{i-1}(q-2m+1)^{j-i-1}(q-2m+2)^{m-j}$.
Combining (i), (i\hspace{-0.5mm}i), and (i\hspace{-0.5mm}i\hspace{-0.5mm}i), we have 
\begin{align*}
|M(\Cc^{(6)}_q)|
&=2(q-2m)^{i-1}(q-2m+1)^{j-i-1}(q-2m+2)^{m-j}\\
&\qquad+(q-2m-1)(q-2m)^{i-1}(q-2m+1)^{j-i-1}(q-2m+2)^{m-j}\\
&=(q-2m)^{i-1}(q-2m+1)^{j-i}(q-2m+2)^{m-j}
\end{align*}
if $q$ is odd.

\subsubsection{The case where $q$ is even}\label{sec-TypeC-rest-x_i=-x_j+1-even}
Let $q$ be even.
Prepare $q-2m$ boxes side by side, $\displaystyle \frac{q}{2}-m$ on the upper side and $\displaystyle \frac{q}{2}-m$ on the lower side, and a circle corresponding to the element $\displaystyle \frac{q}{2}\in\mathbb{Z}_q$ on the right side of the boxes.
Place each of the numbers $1,\dots,j-1,j+1,\dots,m$ in one of the $q-2m$ boxes.
Unlike the case where $q$ is odd, each of the numbers $1,\dots,j-1,j+1,\dots,m$ can be placed in the lower right box.
The circle on the right side of the boxes has no labels since $2x_s\neq 0$ for $s\in [m]$.
The counting in this case is very similar to the case (i\hspace{-0.5mm}i\hspace{-0.5mm}i) in Subsection~\ref{sec-TypeC-rest-x_i=-x_j+1-odd}.

\begin{itemize}
\item The number $i$ can be placed in all boxes.
There are $q-2m$ boxes that can contain the number $i$.

\item Each of the numbers $1,\dots,i-1$ can also be placed in all boxes, and there are $q-2m$ boxes that can contain the numbers $1,\dots,i-1$.

\item There are two ways (before $i$ and after $i$) to place circles with labels $i+1,\dots,j-1$ in the same box as the number $i$, but $i+1,\dots,j-1$ can only be placed before $j$ in the same box as the number $j$.
Therefore, there are $q-2m+1$ choices for the numbers $i+1,\dots,j-1$, since we have $q-2m$ boxes and there are two choices for the same box as $i$.

\item There are two ways to place circles with labels $j+1,\dots,m$ in the same boxes as the numbers $i$ and $j$.
Therefore, there are $q-2m+2$ choices for the numbers $j+1,\dots,m$, since we have $q-2m$ boxes and there are two choices for the same boxes as $i$ and $j$.
\end{itemize}
Therefore, we have $|M(\Cc^{(6)}_q)|=(q-2m)^{i}(q-2m+1)^{j-i-1}(q-2m+2)^{m-j}$ if $q$ is even.

\section{Shi arrangement of type D}\label{sec-TypeD} 
We compute the characteristic quasi-polynomials of the Shi arrangement of type D by providing a bijective relation with the Shi arrangement of type B. 
This is a different way to compute the Shi arrangement of type B and type C. 

We collect some notation on hyperplanes used in this section. 
For $1 \leq i < j \leq m, q \in \ZZ_{>0}, c \in \ZZ$, we define the following:
    \begin{itemize}
        \item Let $H_{i,c} := \{\bm{x} \in \ZZ^m \mid x_i = c\}$.
        \item Let $H_{i,j,c}^- := \{\bm{x} \in \ZZ^m \mid x_i - x_j = c\}$.
        \item Let $H_{i,j,c}^+ := \{\bm{x} \in \ZZ^m \mid x_i + x_j= c\}$.
        \item For an affine hyperplane $H \subset \ZZ^m$ defined by $\alpha_H(\bm{x})=b_H$, where $\alpha_H$ is a linear form with integer coefficients and $b_H\in\ZZ$, let $H[q]:=\{\bm{x}\in\ZZ_q^m\mid \alpha_H(\bm{x})=[b_H]_q\}$, where the coefficients of $\alpha_H$ are reduced modulo $q$.
        \item For a hyperplane arrangement $\Ac$, let $\Ac[q] := \{H[q] \mid H \in \Ac\}$ and $M(\Ac_q):=\ZZ_q^m\setminus\bigcup_{H \in \Ac} H[q]$.
    \end{itemize}

\subsection{A bijection}\label{sec:bij}
First, we prove the following. 
\begin{Them}\label{thm-chara-quasi-TypeD}
    The characteristic quasi-polynomial of the Shi arrangement of type D is 
\begin{align*}
\left|M(\Dc_q)\right|=(q-2m+2)^m
\end{align*}
for any $q\in\mathbb{Z}$ with $q\gg 0$.
\end{Them}
\begin{proof}
We construct a bijection between $M(\Dc_q)$ and $M(\Bc_{q+2})$. 
Once we establish it, we obtain the desired conclusion since we already know $\left|M(\Bc_q)\right|=(q-2m)^m$ by \eqref{eq:typeB_Shi}.

\medskip

For each $\bm{x} \in M(\Dc_q)$ (resp. $\bm{y} \in M(\Bc_{q+2})$), since $\bm{x} \not\in H_{s,s',0}^-[q]$ (resp. $\bm{y} \not\in H_{s,s',0}^-[q+2]$) for any $s,s' \in [m]$ with $s \neq s'$, we have $x_s \neq x_{s'}$ (resp. $y_s \neq y_{s'}$). 
In particular, there is at most one $s \in [m]$ such that $x_s=0$ (resp. $y_s=q+1$). 
By taking this into account, we define the maps $\phi : M(\Dc_q) \to M(\Bc_{q+2})$ and $\psi : M(\Bc_{q+2}) \to M(\Dc_q)$ as follows:
\begin{itemize}
    \item For $\bm{x} \in M(\Dc_q)$, let
    \begin{align*}
        \phi(\bm{x})=
        \begin{cases}
            \bm{x} + \bm{1} + q\bm{e}_u &\text{if $x_u = 0$}, \\
            \bm{x} + \bm{1} &\text{otherwise}, 
        \end{cases}
    \end{align*}
    where $\bm{1} = (1, \ldots, 1)$, $\bm{e}_u$ denotes the $u$-th unit vector, and all resulting vectors are regarded as those in $\ZZ_{q+2}^m$. 
\end{itemize}
\begin{itemize}
    \item For $\bm{y} \in M(\Bc_{q+2})$, let
    \begin{align*}
        \psi(\bm{y})=
        \begin{cases}
            \bm{y} - \bm{1} - q\bm{e}_u &\text{if $y_u = q+1$}, \\
            \bm{y} - \bm{1} &\text{otherwise}, 
        \end{cases}
    \end{align*}
where all resulting points are regarded as vectors in $\ZZ_q^m$. 
\end{itemize}

We check that these maps are well-defined. Once we confirm it, since $\psi \circ \phi$ (resp. $\phi \circ \psi$) is the identity on $M(\Dc_q)$ (resp. $M(\Bc_{q+2})$), we can verify that these are bijective. 

\medskip

\noindent
For the well-definedness for $\phi$, let $\bm{x} \in M(\Dc_q)$.
It is clear that $\phi(\bm{x}) \in \ZZ_{q+2}^m$ by definition. 
\begin{itemize}
\item We see that $\phi(\bm{x}) \notin \bigcup_{s=1}^m (H_{s,0}[q+2] \cup H_{s,1}[q+2])$. 
    \begin{itemize}
        \item In the case where there is no index $u \in [m]$ such that $x_u = 0$, since $\bm{x} \in (\ZZ_q \setminus \{0\})^m$, we obtain $\phi(\bm{x}) = \bm{x} + \bm{1} \in (\ZZ_{q+2} \setminus \{0,1\})^m$. Thus, $\phi(\bm{x}) \notin \bigcup_{s=1}^m (H_{s,0}[q+2] \cup H_{s,1}[q+2])$ holds.
        \item In the case where there is $u \in [m]$ such that $x_u = 0$, the $u$-th entry of $\phi(\bm{x})$ is $q+1$, while each of the other entries is not $0$. Thus, $\phi(\bm{x}) \notin \bigcup_{s=1}^m (H_{s,0}[q+2] \cup H_{s,1}[q+2])$. 
    \end{itemize}
\item We see that $\phi(\bm{x}) \notin \bigcup_{1 \leq s<t \leq m}(H_{s,t,0}^+[q+2] \cup H_{s,t,1}^+[q+2])$.
    \begin{itemize}
        \item Assume $x_s \neq 0$ and $x_t \neq 0$. Since $x_s + x_t \in \ZZ_q \setminus \{0,1\}$, we obtain $(x_s + 1) + (x_t + 1) \in \ZZ_{q+2} \setminus \{0,1\}$. 
        \item Assume $x_s=0$ or $x_t=0$, say, $x_s=0$. Since $x_t \in \ZZ_q \setminus \{0,1\}$, we obtain $(q+1) + (x_t+1) \in \ZZ_{q+2}\setminus\{0,1\}$. 
    \end{itemize}
\item We see that $\phi(\bm{x}) \notin \bigcup_{1 \leq s<t \leq m}(H_{s,t,0}^-[q+2] \cup H_{s,t,1}^-[q+2])$.
    \begin{itemize}
        \item Assume $x_s \neq 0$ and $x_t \neq 0$. 
        Since $x_s - x_t \in \ZZ_q \setminus \{0,1\}$, we have $(x_s - 1) - (x_t - 1) \in \ZZ_{q+2} \setminus \{0,1\}$. 
        \item Assume that $x_s = 0$. (The case $x_t=0$ is similar.) 
        Then $-x_t \in \ZZ_q \setminus \{0,1\}$, i.e., $x_t \in \{1,\ldots,q-2\}$ as elements of $\ZZ$. 
        Thus, we obtain that $(q+1) - (x_t + 1) = q - x_t \in \{2,\ldots,q-1\}$ as $\ZZ$, i.e., $q-x_t \in \ZZ_{q+2} \setminus \{0,1\}$.  
    \end{itemize}
\end{itemize}

\medskip

\noindent
For the well-definedness for $\psi$, let $\bm{y} \in M(\Bc_{q+2})$.
Since $\bm{y} \in (\ZZ_{q+2} \setminus \{0,1\})^m$, we see that $\psi(\bm{y}) \in \ZZ_q^m$.   
\begin{itemize}
\item We see that $\psi(\bm{y}) \notin \bigcup_{1 \leq s<t \leq m}(H_{s,t,0}^+[q] \cup H_{s,t,1}^+[q])$.
    \begin{itemize}
        \item Assume $y_s \neq q+1$ and $y_t \neq q+1$. Since $y_s,y_t,y_s+y_t \in \ZZ_{q+2} \setminus \{0,1\}$, 
        we have $y_s + y_t \in \{4,5,\ldots,2q\} \setminus \{q+2,q+3\}$ as elements of $\ZZ$. 
        Thus, $y_s + y_t - 2 \in \{2,3,\ldots,2q-2\} \setminus \{q,q+1\}$, i.e., $(y_s-1)+(y_t-1)=y_s + y_t - 2 \in \ZZ_q \setminus \{0,1\}$. 
        \item Assume $y_s=q+1$ or $y_t=q+1$, say, $y_s=q + 1$. Then $y_t \in \ZZ_{q+2} \setminus \{0,1,2,q+1\}$ by $y_s+y_t \in \ZZ_{q+2} \setminus \{0,1\}, y_s - y_t \in \ZZ_{q+2} \setminus \{0\}$, and $y_t \in \ZZ_{q+2} \setminus \{0,1\}$. 
        Thus, $y_s + y_t \in \{q+4,q+5,\ldots,2q+1\}$ as $\ZZ$. Hence, we obtain $(y_s-q-1)+(y_t - 1)=y_t-1 \in \{2,\ldots,q-1\}$, 
        i.e., $y_t-1 \in \ZZ_q \setminus \{0,1\}$. 
    \end{itemize}
\item We see that $\psi(\bm{y}) \notin \bigcup_{1 \leq s<t \leq m}(H_{s,t,0}^-[q] \cup H_{s,t,1}^-[q])$.
    \begin{itemize}
        \item Assume $y_s \neq q+1$ and $y_t \neq q+1$. Since $y_s,y_t,y_s - y_t \in \ZZ_{q+2} \setminus \{0,1\}$, 
        we have $y_s-y_t \in \{-q+2,-q+3,\ldots,q-2\} \setminus \{0,1\}$ as $\ZZ$. 
        Thus, we obtain $(y_s -1) - (y_t - 1) =y_s-y_t \in \ZZ_q \setminus \{0,1\}$. 
        \item Assume that $y_s = q+1$. 
        Then $y_t \in \ZZ_{q+2} \setminus \{0,1,2,q+1\}$ by $y_s - y_t \in \ZZ_{q+2} \setminus \{0\}, y_s + y_t \in \ZZ_{q+2} \setminus \{0,1\}$ and $y_t \in \ZZ_{q+2} \setminus \{0,1\}$. Thus, $q+1 - y_t \in \{2,3,\ldots,q-1\}$ as $\ZZ$. 
        Hence, we obtain $-(y_t - 1) \in \ZZ_q \setminus\{0,1\}$.
    \end{itemize}
\end{itemize}

These discussions imply the well-definedness of $\phi$ and $\psi$, as desired. 
\end{proof}

In what follows, we use this bijection by identifying certain restrictions of $M(\Dc_q)$ and $M(\Bc_{q+2})$ (or their analogues). 
We prove the equalities in Theorem~\ref{thm-main-typeD} in a different order, arranged according to the ease of proof as we see it.

\subsection{Characteristic quasi-polynomial of restriction on $\{x_i+x_j=0\}$}\label{sec-quasi-poly-restxi=-xj}
Let $\Dc^{(3)}=\Dc_m^{\{x_i+x_j=0\}}$ for $1\leq i<j \leq m$, and we prove the third equality of Theorem \ref{thm-main-typeD} as follows.
\begin{Them}\label{chara-quasi-del-by-xi=-xj}
We have
\begin{align*}
\left|M(\Dc^{(3)}_q)\right|=
\begin{cases}
    (q-2m+2)^{m-j}(q-2m+3)^{j-i}((q-2m+4)^{i-1}-(q-2m+3)^{i-2}) \;\;&\text{if $q$ is odd}, \\
    (q-2m+2)^{m-j+1}(q-2m+3)^{j-i-1}(q-2m+4)^{i-1} &\text{if $q$ is even} 
\end{cases}
\end{align*}
for any $q\in\ZZ$ with $q\gg 0$.
\end{Them}
\begin{proof}
Let $\Bc^{(3)} = \Bc_m^{\{x_i+x_j=0\}}$.
For our purpose, we show the existence of a bijection between $M(\Dc^{(3)}_q)$ and $M(\Bc^{(3)}_{q+2})$. 
To this end, we use $\phi$ given in the proof of Theorem~\ref{thm-chara-quasi-TypeD}. 
Then it suffices to verify that the map $\phi$ restricted to $M(\Dc^{(3)}_q)$ is well-defined as well as $\psi$ restricted to $M(\Bc^{(3)}_{q+2})$. 
More precisely, our goal is to show that if $\bm{x} \in M(\Dc^{(3)}_q)$ (resp. $\bm{y} \in M(\Bc^{(3)}_{q+2})$), 
then $\phi(\bm{x}) \in M(\Bc^{(3)}_{q+2})$ (resp. $\psi(\bm{y}) \in M(\Dc^{(3)}_q)$). 

First, we check the well-definedness of $\phi$. 
We observe that neither $x_i = 0$ nor $x_j = 0$ holds because this would contradict the conditions $x_i+x_j=0$ and $x_i - x_j \neq 0$. 

Let $\bm{x} \in M(\Dc_q^{(3)})$. It is enough to show that $\phi(\bm{x}) \in H_{i,j,0}^+[q+2]$. 
Since $\bm{x} \in H_{i,j,0}^+[q]$, we have $x_i + x_j = q$ as $\ZZ$. Thus, we obtain $(x_i + 1) + (x_j + 1) = q+2$ as $\ZZ$, i.e., $\phi(\bm{x}) \in H_{i,j,0}^+[q+2]$. 
Hence, the map $\phi$ restricted to $M(\Dc_q^{(3)})$ is well-defined.

Let $\bm{y} \in M(\Bc_{q+2}^{(3)})$. It is enough to show that $\psi(\bm{y}) \in H_{i,j,0}^+[q]$. 
Since $\bm{y} \in H_{i,j,0}^+[q+2]$, we have $y_i + y_j = q+2$ as $\ZZ$. Thus, we obtain $(y_i - 1) + (y_j - 1) = q$ as $\ZZ$, i.e., $\psi(\bm{y}) \in H_{i,j,0}^+[q]$. 
Hence, the map $\psi$ restricted to $M(\Bc_{q+2}^{(3)})$ is well-defined.

Therefore, we obtain that $|M(\Dc_q^{(3)})| = |M(\Bc_{q+2}^{(3)})|$. 
This shows Theorem~\ref{chara-quasi-del-by-xi=-xj} by Theorem~\ref{typeB}. 
\end{proof}

\subsection{Characteristic quasi-polynomial of restriction on $\{x_i-x_j=1\}$}\label{sec-quasi-poly-restxi=xj+1}
Let $\Dc^{(2)}=\Dc_m^{\{x_i-x_j=1\}}$ for $1\leq i<j \leq m$, and we prove the second equality of Theorem \ref{thm-main-typeD} as follows.
\begin{Them}\label{chara-quasi-del-by-xi=xj+1}
We have
\begin{align*}
\left|M(\Dc^{(2)}_q)\right|=(q-2m+2)^{m+i-j}(q-2m+3)^{j-i-1}
\end{align*}
for any $q\in\ZZ$ with $q\gg 0$.
\end{Them}
\begin{proof}
Let $\Bc^{(2)} = \Bc_m^{\{x_i-x_j=1\}}$.
We define the same maps $\phi : M(\Dc^{(2)}_q) \to M(\Bc^{(2)}_{q+2})$ and $\psi  : M(\Bc^{(2)}_{q+2}) \to M(\Dc^{(2)}_q)$ by restricting them. 
%
Then, similarly to Subsection~\ref{sec-quasi-poly-restxi=-xj}, we can verify that the maps $\phi$ and $\psi$ are well-defined. Thus, they are bijective. Hence, we have $|M(\Dc_q^{(2)})| = |M(\Bc_{q+2}^{(2)})|$. This implies Theorem~\ref{chara-quasi-del-by-xi=xj+1} by Theorem~\ref{typeB}. 
\end{proof}

\subsection{Characteristic quasi-polynomial of restriction on $\{x_i-x_j=0\}$}\label{sec-quasi-poly-restxi=xj}
Let $\Dc^{(1)}=\Dc_m^{\{x_i-x_j=0\}}$ for $1\leq i<j \leq m$, and we prove the first equality of Theorem \ref{thm-main-typeD} as follows.
\begin{Them}\label{chara-quasi-del-by-xi=xj}
    We have
    \begin{align*}
        \left|M(\Dc^{(1)}_q)\right|=
        \begin{cases}
            (q-2m+3)^{j-i-1}(q-2m+4)^{m-j}((q-2m+4)^i-(q-2m+3)^{i-1})\\
            -(q-2m+3)^{m-i-1}(q-2m+4)^{i-1}\;\;\;&\text{if $q$ is odd},\\
            (q-2m+3)^{j-i-1}(q-2m+4)^{i-1}((q-2m+4)^{m-j+1}-(q-2m+3)^{m-j}) \\
            -(q-2m+3)^{m-i-1}(q-2m+4)^{i-1} &\text{if $q$ is even}
        \end{cases}
    \end{align*}
for any $q\in\ZZ$ with $q\gg 0$.
\end{Them}
\begin{proof}
Let $\Bc^{(1)} = (\Bc_m \cup \{x_j = -1\})^{\{x_i-x_j = 0\}}$. 
We define the same maps $\phi : M(\Dc^{(1)}_q) \to M(\Bc^{(1)}_{q+2})$ and $\psi  : M(\Bc^{(1)}_{q+2}) \to M(\Dc^{(1)}_q)$ analogously to the previous sections. 
For $\bm{x} \in M(\Dc^{(1)}_q)$, we have $x_j \neq 0$. Indeed, if $x_j = 0$, then $x_i - x_j = 0$ implies $x_i = 0$, which contradicts the condition $x_i + x_j \neq 0$.
Moreover, $\bm{x} \in \ZZ_q^m$ implies $\bm{x} + \bm{1} \in (\ZZ_{q+2} \setminus \{0,q+1\})^m$. 
Therefore, we have $\phi(\bm{x}) \notin H_{j,-1}[q+2]$.
For the remaining conditions (e.g. $\phi(\bm{x}) \not\in \bigcup_{s=1}^m (H_{s,0}[q+2] \cup H_{s,1}[q+2])$), we can verify them straightforwardly as in the previous sections. Thus, they are bijective. Therefore, $\left|M(\Dc^{(1)}_q)\right| = \left|M(\Bc^{(1)}_{q+2})\right|$. 

On the other hand, by using the deletion-restriction formula \eqref{eq:del-res}, we obtain 
\[|M(\Bc^{(1)}_q)| = |M((\Bc_m^{\{x_i-x_j=0\}})_q)| - |M(((\Bc_m \cup \{x_j = -1\})^{\{x_i-x_j=0, x_j=-1\}})_q)|.\] 
It follows from this and Theorem~\ref{typeB} that we can determine $|M(\Bc^{(1)}_q)|$ by computing $|M(((\Bc_m \cup \{x_j = -1\})^{\{x_i-x_j=0, x_j=-1\}})_q)|$. 
By Lemma~\ref{equality-xi=xj} below, we can conclude Theorem~\ref{chara-quasi-del-by-xi=xj}. 
\end{proof}

\begin{Lem}\label{equality-xi=xj}
     Let $\mathcal{\widetilde{B}}^{(1)} = (\Bc_m \cup \{x_j = -1\})^{\{x_i-x_j=0, x_j=-1\}}$. Then we have 
     \begin{align*}
         \left|M(\mathcal{\widetilde{B}}^{(1)}_q) \right|= (q-2m+1)^{m-i-1}(q-2m+2)^{i-1}
     \end{align*}
for any $q\in\ZZ$ with $q\gg 0$.
\end{Lem}
The remaining parts of this subsection are devoted to proving Lemma~\ref{equality-xi=xj}. 

\bigskip

The complement $M(\mathcal{\widetilde{B}}^{(1)}_q)$ is the set of $(x_1,\dots,x_m)\in\ZZ_q^m$ satisfying the following conditions:
\begin{align*}
&x_i=x_j=-1,\\ 
&x_s\neq 0,\ x_s\neq 1\ (s\in [m]),\\
&x_s\neq x_t\ (s,t\in [m],\ s\neq t,\ \{s,t\}\neq \{i,j\}),\\
&x_s\neq x_t+1\ (s,t\in [m],\ s<t,\ (s,t)\neq (i,j)),\\
&x_s\neq -x_t,\ x_s\neq -x_t+1\ (s,t\in [m],\ s\neq t).
\end{align*}

To prove Lemma~\ref{equality-xi=xj}, we fix the condition $x_i=x_j=-1$ and count the number of elements in $(x_1,\dots,x_m)\in M(\mathcal{\widetilde{B}}^{(1)}_q)$ using a modified version of the counting method described in Subsection~\ref{sec-TypeC-counting}.
The condition $x_i=x_j=-1$ means that the circle is labeled with both $i$ and $j$ at the left end of the lower left box.
Thus we consider the numbers $i$ and $j$ as the pair $(i,j)$.

\subsubsection{A proof of Lemma~\ref{equality-xi=xj} when $q$ is odd}\label{sec-quasi-poly-restxi=xj-odd}
Let $q$ be odd.
We create boxes and circles, and take the corresponding element $(x_1,\dots,x_m)\in M(\mathcal{\widetilde{B}}^{(1)}_q)$ as follows:
\begin{itemize}
\item Prepare $q-2m+3$ boxes side by side, $\displaystyle \frac{q+1}{2}-(m-1)$ on the upper side and $\displaystyle \frac{q+1}{2}-(m-1)$ on the lower side.
Place each of the numbers $1,\dots i-1, i+1, \dots,j-1,j+1,\dots,m$ in one of $q-2m+3$ boxes.
The upper left box does not contain a number since $x_s \neq 1$ for any $s\in [m]$.
\end{itemize}
\begin{itemize}
\item The box containing the pair $(i,j)$, which is the lower left box, cannot contain the numbers $i+1,\dots,j-1, j+1,\dots,m$.
Indeed, 
if the clockwise next circle is labeled with $s$ ($i+1 \leq s \leq m$), then this implies $x_j \neq -1$, a contradiction. 
\end{itemize}

We count the number of elements $(x_1,\dots,x_m)\in M(\mathcal{\widetilde{B}}^{(1)}_q)$.
There are $q-2m+2$ boxes that can contain the numbers $1,\dots,i-1$, except the upper left box.
For the remaining numbers, the numbers $i+1,\dots,j-1, j+1,\dots,m$ can be placed anywhere except the upper left and lower left boxes, so $q-2m+1$ boxes can be chosen.
Therefore, the desired number in this case is 
\begin{align*}
        \left|M(\mathcal{\widetilde{B}}^{(1)}_q)\right| = (q-2m+1)^{m-i-1}(q-2m+2)^{i-1}.
\end{align*}
For example, the following boxes and circles correspond to the element $(x_1,x_2,x_3,x_4,x_5)=(7,12,3,12,4)\in M(\mathcal{\widetilde{B}}^{(1)}_q)$ in the case of $m=5,q=13,i=2,j=4$:
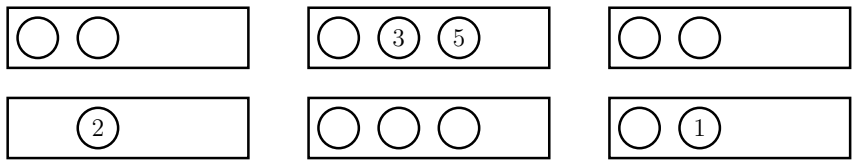
\begin{center}
\scalebox{0.8}{
\begin{tikzpicture}[main/.style = {draw, circle, very thick}] 
\draw [very thick] (0,0) rectangle (4,1); \draw [very thick] (5,0) rectangle (9,1); \draw [very thick] (10,0) rectangle (14,1);
\draw [very thick] (0,1.5) rectangle (4,2.5); \draw [very thick] (5,1.5) rectangle (9,2.5); \draw [very thick] (10,1.5) rectangle (14,2.5);
\node[main](12) at (11.5,0.5) {{\large $1$}};
\node[main](8) at (1.5,0.5) {{\large $2$}};
\node[main](6) at (6.5,2) {{\large $3$}}; 
\node[main](2) at (7.5,2) {{\large $5$}};
\node[main](0) at (0.5,2) {{\color{white}\large $0$}};  
\node[main](0) at (1.5,2) {{\color{white}\large $0$}};  
\node[main](3) at (11.5,2) {\color{white}\large $0$};
\node[main](4) at (5.5,2) {{\color{white}\large $0$}}; 
\node[main](7) at (10.5,2) {{\color{white}\large $0$}};
\node[main](1) at (7.5,0.5) {{\color{white}\large $0$}};  
\node[main](11) at (5.5,0.5) {{\color{white}\large $0$}}; 
\node[main](10) at (6.5,0.5) {{\color{white}\large $0$}};
\node[main](9) at (10.5,0.5) {{\color{white}\large $0$}}; 
\end{tikzpicture}
}
\end{center}

\subsubsection{A proof of Lemma~\ref{equality-xi=xj} when $q$ is even}\label{sec-quasi-poly-restxi=xj-even}
Let $q$ be even.
We count the number of elements $(x_1,\dots,x_m)\in M(\mathcal{\widetilde{B}}^{(1)}_q)$ by dividing the cases by whether there exists an index $k$ such that $\displaystyle x_k=\frac{q}{2}$ or not.

(i)\ Suppose that there is no index $k\in[m]$ such that $\displaystyle x_k=\frac{q}{2}$.
Prepare $q-2m+2$ boxes, $\displaystyle \frac{q}{2}-(m-1)$ on the upper side and $\displaystyle \frac{q}{2}-(m-1)$ on the lower side, and one unlabeled circle corresponding to the element $\displaystyle \frac{q}{2}\in\ZZ_q$ on the right side of the boxes.
Put each of the numbers $1,\dots,i-1,i+1,\dots,j-1,j+1,\dots,m$ and the pair $(i,j)$ in the lower left box. 

As in the case when $q$ is odd, each of the numbers $i+1,\dots,j-1, j+1,\ldots,m$ cannot be placed in the same box as the pair $(i,j)$.

Each of $1,\dots,i-1$ can be placed into one of $q-2m+1$ boxes, except the upper left box.
Then each of the numbers $i+1,\dots,j-1, j+1,\dots,m$ can be placed into one of $q-2m$ boxes except the upper left box and the box containing $(i,j)$, which is the lower left one. 
Therefore, in this case, the desired number is
\begin{align*}
(q-2m)^{m-i-1}(q-2m+1)^{i-1}
\end{align*}

(i\hspace{-0.5mm}i)\ Suppose that there exists an index $k\in[m]$ such that $\displaystyle x_k=\frac{q}{2}$.
Create boxes and circles according to the following rules, and take the corresponding element $(x_1,\dots,x_m)\in M(\mathcal{\widetilde{B}}^{(1)}_q)$.
\begin{itemize}
\item Prepare $q-2m+4$ boxes, $\displaystyle \frac{q}{2}-(m-2)$ on the upper side and $\displaystyle \frac{q}{2}-(m-2)$ on the lower side, and one circle labeled with $k$ on the right side of the boxes.
Place each of the numbers that is not $i$, $j$, or $k$ into one of $q-2m+3$ boxes, except the upper left box.
\item As in the case when $q$ is odd, each of the numbers $i+1,\dots,j-1, j+1,\dots,m$ cannot be placed in the same box as the pair $(i,j)$.
\item The lower right box does not contain any numbers, since $x_k\neq -x_s+1$ for any $s\in [m]$ with $k\neq s$.
The upper right box does not contain any numbers larger than $k$, since $x_k\neq x_t+1$ for any $t \in [m]$ with $k<t$.
\end{itemize}

(i\hspace{-0.5mm}i-1)\ Let $1\leq k\leq i-1$.
Each of the numbers $1,\dots,k-1$ can be placed into one of $q-2m+2$ boxes, except the upper left and the lower right boxes.
Each of the numbers $k+1,\dots,i-1$ can be placed into one of $q-2m+1$ boxes, except the upper left, the upper right, and the lower right boxes.
For the remaining numbers, each of the numbers $i+1,\dots,j-1, j+1,\dots,m$ can be placed into one of $q-2m$ boxes, except the upper left, the upper right, and the lower right boxes and the box containing $(i,j)$.
Therefore, in this case, the desired number is
\[
(q-2m)^{m-i-1}(q-2m+1)^{i-k-1}(q-2m+2)^{k-1}
\]

(i\hspace{-0.5mm}i-2)\ Let $i+1\leq k\leq j-1$.
Each of the numbers $1,\dots,i-1$ can be placed into one of $q-2m+2$ boxes, except the upper left and the lower right boxes.
For the remaining numbers, each of the numbers $i+1,\dots,k-1$ can be placed into one of $q-2m+1$ boxes, except the upper left and the lower right boxes and the box containing $(i,j)$.
Each of the numbers $k+1,\dots,j-1, j+1,\dots,m$ can be placed into one of $q-2m$ boxes, except the upper left, the upper right, and the lower right boxes and the box containing $(i,j)$.
Therefore, in this case, the desired number is
\begin{align*}
(q-2m)^{m-k-1}(q-2m+1)^{k-i-1}(q-2m+2)^{i-1}.
\end{align*}

(i\hspace{-0.5mm}i-3)\ Let $j+1\leq k\leq m$.
Each of the numbers $1,\dots,i-1$ can be placed into one of $q-2m+2$ boxes, except the upper left and the lower right boxes.
Each of the numbers $i+1,\dots,j-1, j+1,\dots,k-1$ can be placed into one of $q-2m+1$ boxes, except the upper left and the lower right boxes and the box containing $(i,j)$.
Finally, each of the numbers $k+1,\dots,m$ can be placed into one of $q-2m$ boxes, except the upper left, the upper right, the lower right boxes and the box containing $(i,j)$.
Therefore, in this case, the desired number is
\begin{align*}
(q-2m)^{m-k}(q-2m+1)^{k-i-2}(q-2m+2)^{i-1}.
\end{align*}

From the discussion above, we have that
\begin{align*}
\left|M(\mathcal{\widetilde{B}}^{(1)}_q)\right|
=&\,T^{m-i-1}(T+1)^{i-1} + \sum_{k=1}^{i-1}T^{m-i-1}(T+1)^{i-k-1}(T+2)^{k-1}\\
&+\sum_{k=i+1}^{j-1}T^{m-k-1}(T+1)^{k-i-1}(T+2)^{i-1} + \sum_{k=j+1}^{m}T^{m-k}(T+1)^{k-i-2}(T+2)^{i-1}\\
=&\,T^{m-i-1}(T+1)^{i-1} + T^{m-i-1}\left((T+2)^{i-1}-(T+1)^{i-1}\right)\\
&+T^{m-j}(T+2)^{i-1}\left((T+1)^{j-i-1}-T^{j-i-1}\right)+(T+1)^{j-i-1}(T+2)^{i-1}\left((T+1)^{m-j}-T^{m-j}\right)\\
=&\,(T+1)^{m-i-1}(T+2)^{i-1},
\end{align*}
where $T:=q-2m$.

\subsection{Characteristic quasi-polynomial of restriction on $\{x_i+x_j=1\}$}\label{sec-quasi-poly-restxi=-xj+1}
Let $\Dc^{(4)}=\Dc_m^{\{x_i+x_j=1\}}$ for $1\leq i<j \leq m$, and we prove the fourth equality of Theorem \ref{thm-main-typeD} as follows.
\begin{Them}\label{chara-quasi-del-by-xi=-xj+1}
We have
\begin{align*}
\left|M(\Dc^{(4)}_q)\right|=
\begin{cases}
    (q-2m+2)^{i-1}(q-2m+3)^{j-i}(q-2m+4)^{m-j} - \\
    (q-2m+2)^{i-1}(q-2m+3)^{m-i-1} &\text{if $q$ is odd}, \\
    (q-2m+2)^{i-1}(q-2m+3)^{j-i-1}((q-2m+4)^{m-j+1}-(q-2m+3)^{m-j}) \\
    - (q-2m+2)^{i-1}(q-2m+3)^{m-i-1}&\text{if $q$ is even}
\end{cases}
\end{align*}
for any $q\in\ZZ$ with $q\gg 0$.
\end{Them}
\begin{proof}
Let $\Bc^{(4)} = (\Bc_m \cup \{x_j = -1\})^{\{x_i+x_j=1\}}$.
We define the maps $\phi : M(\Dc^{(4)}_q) \to M(\Bc^{(4)}_{q+2})$ and $\psi : M(\Bc^{(4)}_{q+2}) \to M(\Dc^{(4)}_q)$ in the same way. 
For $\bm{x} \in M(\Dc^{(4)}_q)$, we have $x_j \neq 0$. Indeed, if $x_j = 0$, then $x_i + x_j = 1$ implies $x_i = 1$, which contradicts the condition $x_i - x_j \neq 1$. Moreover, $\bm{x} \in \ZZ_q^m$ implies $\bm{x} + \bm{1} \in (\ZZ_{q+2} \setminus \{0,q+1\})^m$. Therefore, we have $\phi(\bm{x}) \notin H_{j,-1}[q+2]$. 
We can also  straightforwardly verify the remaining conditions for the well-definedness of $\phi$. 
Hence, this gives $\left|M(\Dc^{(4)}_q)\right| = \left|M(\Bc^{(4)}_{q+2})\right|$. 

On the other hand, by the deletion-restriction formula, we obtain 
\[|M(\Bc^{(4)}_q)| = |M((\Bc_m^{\{x_i+x_j=1\}})_q)| - |M(((\Bc_m \cup \{x_j = -1\})^{\{x_i+x_j=1, x_j=-1\}})_q)|.\] 

Therefore, we complete the proof of Theorem~\ref{chara-quasi-del-by-xi=-xj+1} by proving Lemma~\ref{equality-xi=-xj+1} below. 
\end{proof}
\begin{Lem}\label{equality-xi=-xj+1}
    Let $\mathcal{\widetilde{B}}^{(4)} = (\Bc_m \cup \{x_j = -1\})^{\{x_i+x_j=1, x_j=-1\}}$.
    Then we have 
    \begin{align*}
        \left|M(\mathcal{\widetilde{B}}^{(4)}_q)\right| = (q-2m)^{i-1}(q-2m+1)^{m-i-1}
    \end{align*}
for any $q \in \ZZ$ with $q \gg 0$. 
\end{Lem}
The remaining parts of this subsection are devoted to proving Lemma~\ref{equality-xi=-xj+1}. 

\bigskip

The complement $M(\mathcal{\widetilde{B}}_q^{(4)})$ is the set of $(x_1,\dots,x_m)\in\ZZ_q^m$ satisfying the following conditions:
\begin{align*}
&x_i=2, x_j = -1\\
&x_s\neq 0,\ x_s\neq 1\ (s\in [m]),\\
&x_s\neq x_t\ (s,t\in [m],\ s\neq t),\\
&x_s\neq x_t+1\ (s,t\in [m],\ s<t),\\
&x_s\neq -x_t,\ x_s\neq -x_t+1\ (s,t\in [m],\ s\neq t,\ \{s,t\}\neq \{i,j\}).
\end{align*}
We fix the condition $x_i=2, x_j = -1$ and count the number of elements in $(x_1,\dots,x_m)\in M(\mathcal{\widetilde{B}}^{(4)}_q)$ using a modified version of the counting method. 

\subsubsection{A proof of Lemma~\ref{equality-xi=-xj+1} when $q$ is odd}
Let $q$ be odd.
The condition $x_i=2$ means that the circle labeled with $i$ comes after two unlabeled circles in the upper left box, 
and the condition $x_j=-1$ means that the circle labeled with $j$ is placed at the left end of the lower left box. 
\begin{itemize}
\item Prepare $q-2m+1$ boxes side by side, $\displaystyle \frac{q+1}{2}-m$ on the upper side and $\displaystyle \frac{q+1}{2}-m$ on the lower side. Place each of the numbers $1,\dots, i-1, i+1, \dots, j-1, j+1,\ldots,m$ in one of $q-2m+1$ boxes. 

%
%
\item There are $q-2m$ boxes that can contain the numbers $1,\ldots,i-1$, which are the boxes except the upper left one. 
\item Since $i$ is placed in the upper left box, the number $k$ with $i<k<j$ can be placed there. 
Hence, there are $q-2m+1$ boxes that can contain the numbers $i+1,\ldots,j-1$. 
\item The number $k$ with $k>j$ can be placed in the lower left box since there is an unlabeled circle just before the circle labeled with $j$. 
In addition, we can place $k$ at the right side of the circle labeled with $i$ in the upper left box. 
Note that the left side of the circle labeled with $i$ in the upper left box is forbidden for $k$ to be placed. 
Indeed, if $k$ is placed there, then we have $x_i \neq 2$, a contradiction. 
Hence, there are $q-2m+1$ boxes that can contain the numbers $j+1,\ldots,m$. 
\end{itemize}

There are $q-2m$ boxes that can contain the numbers $1,\dots,i-1$, except the upper left box.
For the remaining numbers, the numbers $i+1,\dots,j-1, j+1,\dots,m$ can be placed anywhere, so $q-2m+1$ boxes can be chosen.
Therefore, the desired number in this case is 
\begin{align*}
    \left|M(\mathcal{\widetilde{B}}^{(4)}_q)\right| = (q-2m)^{i-1}(q-2m+1)^{m-i-1}. 
\end{align*}
For example, the following boxes and circles correspond to the element $(x_1,x_2,x_3,x_4,x_5)=(4,2,8,14,9)\in M(\widetilde{\Bc}^{(4)}_q)$ in the case of $m=5,q=15,i=2,j=4$:
\begin{center}
\scalebox{0.8}{
\begin{tikzpicture}[main/.style = {draw, circle, very thick}] 
\draw [very thick] (0,0) rectangle (4,1); \draw [very thick] (5,0) rectangle (9,1); \draw [very thick] (10,0) rectangle (14,1);
\draw [very thick] (0,1.5) rectangle (4,2.5); \draw [very thick] (5,1.5) rectangle (9,2.5); \draw [very thick] (10,1.5) rectangle (14,2.5);
\node[main](12) at (12.5,0.5) {{\large $3$}};
\node[main](5) at (6.5,2) {{\large $1$}};
\node[main](6) at (1.5,2) {{\color{white}\large $0$}}; 
\node[main](6) at (2.5,2) {{\large $2$}}; 
\node[main](10) at (1.5,0.5) {{\large $4$}};
\node[main](2) at (11.5,0.5) {{\large $5$}};
\node[main](0) at (0.5,2) {{\color{white}\large $0$}};  
\node[main](3) at (12.5,2) {\color{white}\large $0$};
\node[main](4) at (5.5,2) {{\color{white}\large $0$}}; 
\node[main](7) at (10.5,2) {{\color{white}\large $0$}};
\node[main](1) at (11.5,2) {{\color{white}\large $0$}};  
\node[main](11) at (5.5,0.5) {{\color{white}\large $0$}}; 
\node[main](9) at (10.5,0.5) {{\color{white}\large $0$}}; 
\node[main](9) at (6.5,0.5) {{\color{white}\large $0$}}; 
\node[main](9) at (2.5,0.5) {{\color{white}\large $0$}}; 
\end{tikzpicture}
}
\end{center}

\subsubsection{A proof of Lemma~\ref{equality-xi=-xj+1} when $q$ is even}\label{sec-count-method-x_i+x_j=1-even}
Let $q$ be even.

(i)\ Suppose that there is no index $k\in[m]$ such that $\displaystyle x_k=\frac{q}{2}$.
Prepare $q-2m$ boxes, $\displaystyle \frac{q}{2}-m$ on the upper side and $\displaystyle \frac{q}{2}-m$ on the lower side, 
and one circle corresponding to the element $\displaystyle \frac{q}{2}\in\ZZ_q$ on the right side of the boxes. 
In this case, we see that the possible ways to place those numbers are almost the same as the case when $q$ is odd, but we may just replace $q$ with $q-1$ because of the number of boxes. 
Therefore, in this case, the desired number is 
\[
(q-2m-1)^{i-1}(q-2m)^{m-i-1}.
\]

\bigskip

(i\hspace{-0.5mm}i)\ Suppose that there exists an index $k\in[m]$ such that $\displaystyle x_k=\frac{q}{2}$.
Create boxes and circles according to the following rules and take the corresponding element $(x_1,\dots,x_m)\in M(\mathcal{\widetilde{B}}_q^{(4)})$.
\begin{itemize}
\item Prepare $q-2m+2$ boxes, $\displaystyle \frac{q}{2}-(m-1)$ on the upper side and $\displaystyle \frac{q}{2}-(m-1)$ on the lower side, 
and one circle labeled with $k$ on the right side of the boxes. 
\item The lower right box does not contain any numbers, except as previously noted. 
Indeed, if there is a labeled circle in the lower right box in such cases, then the clockwise next circle from the opposite circle of the rightmost circle, say labeled with $s$, is the circle labeled with $k$. 
In other words, $x_k=-x_s+1$ holds, a contradiction. 
\item The upper right box does not contain any numbers larger than $k$.
Indeed, if the upper right box contains a number larger than $k$, then the rightmost circle in the upper right box is a labeled circle, say labeled with $s$ ($k<s$).
The circle clockwise preceding the circle labeled with $k$ has the label greater than $k$.
In other words, $x_k=x_s+1$ holds, a contradiction.
\end{itemize}
We divide the discussions into which range $k$ belongs. 
\begin{itemize}
\item[(i\hspace{-0.5mm}i-1)] Let $1 \leq k \leq i-1$. 
\begin{itemize}
\item There are $q-2m$ boxes that can contain the numbers $1,\ldots,k-1$, which are the boxes except the upper left and lower right ones. 
\item There are $q-2m-1$ boxes that can contain the numbers $k+1,\ldots,i-1$, which are the boxes except the upper left, the upper right and the lower right ones. 
\item Similarly to the case when $q$ is odd, 
there are $q-2m$ boxes that can contain $i+1,\ldots,j-1, j+1,\ldots,m$. 
\end{itemize}
Hence the desired number is 
\[
(q-2m-1)^{i-k-1}(q-2m)^{m+k-i-2}
\] 
in the case of $1 \leq k \leq i-1$. 

\item[(i\hspace{-0.5mm}i-2)] Let $i+1 \leq k \leq j-1$. 
\begin{itemize}
\item There are $q-2m$ boxes that can contain the numbers $1,\ldots,i-1$, which are the boxes except the upper left and lower right ones. 
\item Similarly, there are $q-2m+1$ boxes that can contain the numbers $i+1,\ldots,k-1$. 
\item There are $q-2m$ boxes that can contain the numbers $k+1,\ldots,j-1, j+1,\ldots,m$. 
\end{itemize}
Hence the desired number is 
\[
(q-2m)^{m+i-k-2}(q-2m+1)^{k-i-1}
\]
in the case of $i+1 \leq k \leq j-1$. 

\item[(i\hspace{-0.5mm}i-3)] Let $j+1 \leq k \leq m$. 
\begin{itemize}
\item There are $q-2m$ boxes that can contain the numbers $1,\ldots,i-1$. 
\item There are $q-2m+1$ boxes that can contain the numbers $i+1,\ldots,j-1, j+1,\ldots,k-1$. 
\item There are $q-2m$ boxes that can contain the numbers $k+1,\ldots,m$. 
\end{itemize}
Hence the desired number is 
\[
(q-2m)^{m+i-k-1}(q-2m+1)^{k-i-2}
\] 
in the case of $j+1 \leq k \leq m$. 
\end{itemize}

From the discussion above, we have that
\begin{align*}
\left|M(\mathcal{\widetilde{B}}^{(4)}_q)\right|
=&\,(T-1)^{i-1}T^{m-i-1} + \sum_{k=1}^{i-1}(T-1)^{i-k-1}T^{m+k-i-2}\\
&+\sum_{k=i+1}^{j-1}T^{m+i-k-2}(T+1)^{k-i-1} + \sum_{k=j+1}^{m}T^{m+i-k-1}(T+1)^{k-i-2}\\
=&\,(T-1)^{i-1}T^{m-i-1} + T^{m-i-1}\left(T^{i-1}-(T-1)^{i-1}\right)\\
&+T^{m+i-j-1}\left((T+1)^{j-i-1}-T^{j-i-1}\right)+T^{i-1}(T+1)^{j-i-1}\left((T+1)^{m-j}-T^{m-j}\right)\\
=&\,T^{i-1}(T+1)^{m-i-1},
\end{align*}
where $T:=q-2m$.


\section{Proofs of corollaries}\label{sec:corollaries}
In this section, we prove Corollaries~\ref{period-collapse-1}, \ref{period-collapse-2} and \ref{Cor-latticeBC} by using Theorems \ref{thm-main-typeC} and \ref{thm-main-typeD}. 

For the proofs of Corollaries~\ref{period-collapse-1} and \ref{period-collapse-2}, we may check the period collapse of the restrictions instead of the deletions. 
\begin{proof}[Proof of Corollary~\ref{period-collapse-1}]
(1) If $H=\{2x_i=0\}$ or $H=\{2x_i=1\}$, then we can easily see that $|M((\mathcal{C}_m^H)_q)|$ never becomes a polynomial by Theorem~\ref{thm-main-typeC}. 
Moreover, if $H=\{x_i-x_j=0\}$ ($1 \leq i < j \leq m$) or $H=\{x_i-x_j=1\}$ ($1 \leq i < j \leq m)$, then $|M((\mathcal{C}_m^H)_q)|$ becomes a polynomial as Theorem~\ref{thm-main-typeC} shows. 

Let $H=\{x_i+x_j=0\}$ for $1 \leq i < j \leq m$. Suppose that $|M((\mathcal{C}_m^H)_q)|$ becomes a polynomial. Then the equality
\begin{align*}
    (q-2m)^{m-j}(q-2m+1)^{j-i}(q-2m+2)^{i-1} &= (q-2m)^{m-j}(q-2m+1)^{j-i-1}(q-2m+2)^{i} \\
    \Longleftrightarrow q-2m+1 &= q-2m+2
\end{align*}
should be satisfied, a contradiction.

Let $H=\{x_i+x_j=1\}$ for $1 \leq i < j \leq m$. Suppose that $|M((\mathcal{C}_m^H)_q)|$ becomes a polynomial. Then the equality 
\begin{align*}
    (q-2m)^{i-1}(q-2m+1)^{j-i}(q-2m+2)^{m-j} &= (q-2m)^{i}(q-2m+1)^{j-i-1}(q-2m+2)^{m-j} \\
    \Longleftrightarrow q-2m+1 &= q-2m 
\end{align*}
should be satisfied, a contradiction.

\noindent
(2) Fix $H \in \Dc_m$. From the proofs of Theorems~\ref{chara-quasi-del-by-xi=-xj}, \ref{chara-quasi-del-by-xi=xj+1}, \ref{chara-quasi-del-by-xi=xj} and \ref{chara-quasi-del-by-xi=-xj+1}, we see that the characteristic quasi-polynomial of $\Dc_m \setminus \{H\}$ becomes a polynomial if and only if that of $\Bc_m \setminus \{H\}$ does. Therefore, by \cite[Corollary 1.4]{HN2024}, the characteristic quasi-polynomial of $\Dc_m \setminus \{H\}$ becomes a polynomial if and only if $H$ is one of the following:
\begin{itemize}
    \item $H=\{x_i-x_{m+1-i}=0\}$ for $1\leq i\leq\lfloor m/2\rfloor$;
    \item $H=\{x_i-x_j=1\}$ for $1 \leq i<j \leq m$; 
    \item $H=\{x_1+x_j=0\}$ for $2 \leq j \leq m$; 
    \item $H=\{x_i+x_m=1\}$ for $1 \leq i \leq m-1$.  
\end{itemize}
\end{proof}

\begin{proof}[Proof of Corollary~\ref{period-collapse-2}]
(1) Let $\Cc=\Cc_m$.
For the distinct parallel hyperplanes $H,H'\in\Cc$, we have $H_q\cap H'_q=\varnothing$ for every integer $q\geq 2$.
It follows that
\begin{align*}
    M((\Cc\setminus\{H,H'\})_q)=M(\Cc_q)\sqcup M((\Cc^H)_q)\sqcup M((\Cc^{H'})_q).
\end{align*}
Hence,
\begin{align*}
    |M((\Cc\setminus\{H,H'\})_q)|=|M(\Cc_q)|+|M((\Cc^H)_q)|+|M((\Cc^{H'})_q)|.
\end{align*}
Since the characteristic quasi-polynomial of $\Cc$
is a polynomial, that of $\Cc\setminus\{H,H'\}$
is a polynomial if and only if the sum of the
quasi-polynomials counting $|M((\Cc^H)_q)|$
and $|M((\Cc^{H'})_q)|$ is a polynomial. 

In the case of the pair $H=\{2x_i=0\}$ and $H^{\prime}=\{2x_i=1\}$ for $1 \leq i \leq m$, we see the following:
\begin{align*}
    &(q-2m)^{m-i}(q-2m+1)^{i-1} + (q-2m)^{i-1}(q-2m+1)^{m-i} = 2(q-2m)^{m-i}(q-2m+1)^{i-1} \\
    \Longleftrightarrow \ &(q-2m)^{i-1}(q-2m+1)^{m-i} = (q-2m)^{m-i}(q-2m+1)^{i-1} \\
    \Longleftrightarrow \ &m-i = i-1 \Longleftrightarrow 2i = m+1
\end{align*}

In the case of the pair $H=\{x_i-x_j=0\}$ and $H^{\prime}=\{x_i-x_j=1\}$ for $1 \leq i<j \leq m$, it is clear that $|M((\Cc^H)_q)| + |M((\Cc^{H^{\prime}})_q)|$ becomes a polynomial.

In the case of the pair $H=\{x_i+x_j=0\}$ and $H^{\prime}=\{x_i+x_j=1\}$ for $1 \leq i<j \leq m$, we see the following:
\begin{align*}
    &(q-2m)^{m-j}(q-2m+1)^{j-i}(q-2m+2)^{i-1} + (q-2m)^{i-1}(q-2m+1)^{j-i}(q-2m+2)^{m-j} \\
    &= (q-2m)^{m-j}(q-2m+1)^{j-i-1}(q-2m+2)^{i} + (q-2m)^{i}(q-2m+1)^{j-i-1}(q-2m+2)^{m-j} \\
    \Longleftrightarrow \ &(q-2m)^{m-j}(q-2m+1)^{j-i-1}(q-2m+2)^{i-1} = (q-2m)^{i-1}(q-2m+1)^{j-i-1}(q-2m+2)^{m-j} \\
    \Longleftrightarrow \ &i-1 = m-j \Longleftrightarrow i+j = m+1, 
\end{align*}
as desired.

\noindent
(2) Fix $H,H^{\prime} \in \Dc_m$ which are parallel each other. As in Corollary~~\ref{period-collapse-1}, the characteristic quasi-polynomial of $\Dc_m \setminus \{H,H^{\prime}\}$ becomes a polynomial if and only if that of $\Bc_m \setminus \{H,H^{\prime}\}$ does. Therefore, by \cite[Corollary 1.6]{HN2024}, the characteristic quasi-polynomial of $\Dc_m \setminus \{H,H^{\prime}\}$ becomes a polynomial if and only if one of the following is satisfied: 
\begin{itemize}
    \item $H=\{x_i-x_{m+1-i}=0\}$ and $H^{\prime}=\{x_i-x_{m+1-i}=1\}$ for $1\leq i\leq\lfloor m/2\rfloor$; 
    \item $H=\{x_i+x_{m+1-i}=0\}$ and $H^{\prime}=\{x_i+x_{m+1-i}=1\}$ for $1\leq i\leq\lfloor m/2\rfloor$. 
\end{itemize}
\end{proof}

For an affine arrangement $\mathcal{A}$ over a field, the intersection poset $L(\mathcal{A})$ is defined by $L(\mathcal{A}):=\{\bigcap_{H\in\mathcal{B}} H\mid \mathcal{B}\subseteq \mathcal{A},\ \bigcap_{H\in\mathcal{B}} H\neq\emptyset\}$.
Then $L(\mathcal{A})$ forms a poset by reverse inclusion.
We say that posets $(P,\le_P)$ and $(Q,\le_Q)$ are isomorphic if there exists a bijection $\varphi \colon P \to Q$ such that for all $X,Y \in P$, $x \le_P y$ is equivalent to $\varphi(x) \le_Q \varphi(y)$.
\begin{proof}[Proof of Corollary \ref{Cor-latticeBC}]
Assume to the contrary that the intersection posets of $\Bc_m$ and $\Cc_m$ were order isomorphic. Then the intersection poset of the restriction of each hyperplane in $\Bc_m$ would also be isomorphic to that of the corresponding restriction in $\Cc_m$.
The characteristic polynomial is determined solely by the intersection poset.
Although Theorem~\ref{typeB} shows that the polynomial $(T+1)^{j-3}(T+2)^{m-j}(T^2+3T+3)$ appears as the characteristic polynomial of the restriction of ${x_2-x_j=0}$ in $\Bc_m$ for some $2<j \leq m$, every characteristic polynomial of the restriction in $\Cc_m$ decomposes only into linear factors $(T+a)$ with $a\in \{0,1,2\}$, yielding a contradiction.
\end{proof}

\end{document}